\documentclass[12pt,reqno]{amsart}
\usepackage[english]{babel}
\usepackage[a4paper,twoside]{geometry}
\usepackage{microtype}
\usepackage{enumerate}
\usepackage{csquotes}
\usepackage{longtable}
\usepackage{bbm}

\usepackage{amsmath}
\usepackage{amsthm}
\usepackage{amssymb}
\usepackage{mathrsfs} 
\usepackage{eucal}
\usepackage{mathtools}

\usepackage{graphicx}
\usepackage{tikz}
\usetikzlibrary{cd}
\usepackage{version}

\usepackage{enumitem}
\setlist[description]{%
  style=unboxed,
  leftmargin=32pt,
}

\usepackage{xspace}
\patchcmd{\subsection}{-.5em}{.5em}{}{}
\usepackage[hyperfootnotes=false]{hyperref} 
\usepackage{color}
\hypersetup{
	colorlinks = true,
	linkcolor = {blue},
	urlcolor = {red},
	citecolor = {blue}
      }

\newcounter{results}[section]

\theoremstyle{plain}
\newtheorem{theorem}[results]{Theorem}

\newtheorem{lemma}[results]{Lemma}
\newtheorem{proposition}[results]{Proposition}

\theoremstyle{remark}

\theoremstyle{definition}

\numberwithin{equation}{section}

\makeatletter
\@namedef{subjclassname@2020}{ \textup{2020}  Mathematics Subject Classification}
\makeatother

\newcommand{\R}{{\mathbb R}}
\newcommand{\N}{{\mathbb N}}
\newcommand{\la}{\langle}
\newcommand{\ra}{\rangle}
\newcommand{\de}{{\rm d}}
\newcommand{\haz}{\widehat}
\newcommand{\weakto}{\rightharpoonup}

\newcommand{\Span}{\mathrm{span}}
\newcommand{\zz}{{\mbox{\boldmath$z$}}}
\newcommand{\xx}{{\boldsymbol x}} 
\newcommand{\ddelta}{\boldsymbol{\delta}}
\newcommand{\epsi}{\varepsilon}
\renewcommand{\ne}{{\nu\epsi}}

\title[Time-fractional doubly nonlinear equation]{Global well-posedness for a time-fractional \\doubly
  nonlinear equation}

\author{Goro Akagi}
\address[Goro Akagi]{Mathematical Institute and Graduate School of Science, Tohoku University, Sendai 980-8578, Japan}
\email{goro.akagi@tohoku.ac.jp}

\author{Giacomo Enrico Sodini}
\address[Giacomo Enrico Sodini]{Faculty of Mathematics, University of Vienna, Oskar-Morgenstern-Platz 1, A-1090 Vienna, Austria}
\email{giacomo.sodini@univie.ac.at}

\author{Ulisse Stefanelli}
\address[Ulisse Stefanelli]{Faculty of Mathematics, University of
  Vienna, Oskar-Morgenstern-Platz 1, A-1090 Vienna, Austria \& Vienna Research Platform on Accelerating Photoreaction Discovery, University of Vienna, W\"ahringerstrasse 17, A-1090 Wien, Austria.}
\email{ulisse.stefanelli@univie.ac.at}


\begin{document}

\begin{abstract} We consider a time-fractional parabolic equation of doubly
  nonlinear type, featuring nonlinear terms both
  inside and outside the differential operator in time. The main
  nonlinearities are maximal monotone graphs, without restrictions on
  the growth.
In addition, a Lipschitz continuous perturbation is
considered. The existence of global weak solutions is obtained via a
regularization and Galerkin approximation method.  Uniqueness is also discussed under some additional assumptions.
\end{abstract}

\maketitle
%\tableofcontents
%\thispagestyle{empty}

\section{Introduction}\label{sec:intro}

This paper is concerned with the nonlinear time-fractional parabolic
equation
\begin{equation}
  \label{eq:0}
  \partial^\theta_t (\alpha(u) - \alpha(u_0)) - \Delta u + \beta(u)\ni g(x,t,u).
\end{equation}
Here, the symbol $\partial^\theta_t$ stands for the {\it Riemann--Liouville
  fractional derivative} of order $\theta\in (0,1)$ which is defined by
$$\partial^\theta_t f(t):=\frac{1}{\Gamma(1{-}\theta)}\frac{\de}{\de
  s} \int_0^t (t{-}s)^{-\theta}f(s)\, \de t\quad \forall t>0$$
where $\Gamma$ is the Euler  $\Gamma$-function. Hence if $\alpha(u)
|_{t = 0} = \alpha(u_0)$, then the Riemann--Liouville derivative
$\partial^\theta_t (\alpha(u) - \alpha(u_0))$ coincides of the {\it Caputo derivative} of $\alpha(u)$ of order $\theta \in (0,1)$ in a weak sense. 

Equation \eqref{eq:0} is posed in the space-time cylinder $Q:=\Omega
\times (0,T)$, where $T>0$ is a final time and $\Omega$ is a nonempty,
open, connected, bounded  set  of $\R^n$ with Lipschitz boundary
$\partial \Omega$. The nonlinearities $\alpha, \, \beta : \R\to 2^\R
$ are maximal monotone graphs, with no growth restrictions, whereas $g(x,t,\cdot)$ is a Lipschitz continuous function. The initial value
$\alpha(u_0)$ is given and the  equation  is complemented by
homogeneous Dirichlet boundary conditions.

The occurrence of nonlinearities both inside and outside the
differential operator in time $\partial_t^\theta$ qualifies equation
\eqref{eq:0} as of {\it doubly nonlinear} type. Doubly nonlinear
problems occur frequently in connection with applications, especially
in relation to nonlinear diffusion phenomena. In the nonfractional
case, for different choices of the map $\alpha $, equation
\eqref{eq:0} may arise in connection to the two-phase Stefan problem
(for $\alpha={\rm id} +H$, where $H$ is the Heaviside graph), the
porous-medium equation (for $\alpha(u)=|u|^{p -2}u$ for some $p \in
(1,2)$) and the Hele--Shaw model (for $\alpha=H$), see~\cite{Visintin96} for a detailed
discussion.

Also motivated by their relevance in applications, doubly nonlinear
problems are a mainstay in evolution-equation theory. In the
classical, nonfractional case, existence results for doubly nonlinear
equations can be traced back  to Grange--Mignot~\cite{Grange-Mignot72}.
Besides the classical references by {Barbu}~\cite{B-dn},
{DiBenedetto--Showalter}~\cite{DiBenedetto-Show81},  {Alt--Luckhaus}~\cite{Alt-Luckhaus}, and {Bernis}~\cite{Bernis88}, one may also
refer to 
\cite{Aizicovici-Hokkanen04b,Aizicovici-Hokkanen04,Akagi06,Akagi11,Gajewsky-Skrypnik04,Hokkanen91,Hokkanen92b,Hokkanen92,MW-dn,vol,fico,ganzo},
among many others. See
also~\cite{Akagi14,be} for some global variational approaches  and~\cite{Akagi-Schimperna,Boegelein,Huaroto,Kato,Moring} for a collection of recent results.

To the best of our knowledge, existence for time-fractional parabolic
doubly nonlinear equations has not been proved yet. This note is
intended to fill this gap. Our main result is the well-posedness of a
variational formulation of equation \eqref{eq:0}, where the elliptic
operator is considered in the classical weak sense. Note that, in case
$\alpha$ is linear, existence of solutions to equation \eqref{eq:0}
has already been obtained in~\cite{Akagi19}, see also 
\cite{LiS23,Akagi24}. In fact, the specific case of porous media
has already been considered in~\cite{Ploci1,Ploci2,Ploci3}. In
addition, a number of results on long-time behavior and decay is
available, see~\cite{Affili,Dipierro,Vergara1,Vergara2}, where
nonetheless existence of solutions is just assumed. Furthermore,
in~\cite{BII}, the asymptotic behavior of solutions for such time-fractional porous medium equations has also been studied in detail.  The reader is also referred to~\cite{Giga,Topp} for some alternative approaches based on viscosity solutions.

To prove existence of a solution to \eqref{eq:0}
we proceed by regularization and
approximation. The maximal monotone graphs $\alpha$ and
$\beta$ are replaced by their Yosida approximations and $\alpha$ is made strongly
monotone by adding a multiple of the identity  to it.  Such  a  regularized problem is then tackled by a Galerkin discretization
method. A suite of a-priori estimates allows the passage to the limit,
first in the Galerkin approximation and then in the regularization,
bringing to the existence proof. 
In the classical nonfractional case, a pivotal tool in the analysis of
doubly nonlinear problems driven by  subdifferential  operators is the availability of a local chain-rule equality, which
eventually allows for obtaining a-priori estimates. In the time-fractional
setting,  the local chain-rule equality is replaced by a nonlocal
chain-rule inequality, which makes the analysis more
challenging.

Uniqueness is generally not to be expected for doubly nonlinear
problems, see~\cite{DiBenedetto-Show81}. We are nonetheless in the position of providing  a 
uniqueness result under some stronger assumptions, namely, for $\alpha$ strongly monotone and $\beta$ 
Lipschitz continuous. 

The plan of the paper is as follows. We present our setting and
state the well-posedness result, i.e., Theorem \ref{thm:main}, in
Section \ref{sec:main1}. The uniqueness of solutions is proved in
Section \ref{sec:uniqueness}. The existence proof is then presented in
Section \ref{sec:existence}. A regularization  and Galerkin
approximation of the problem is introduced in Sections 
 \ref{sec:regularization}--\ref{sec:galerkin}. We then prove a-priori
estimates and pass to the
limit in the Galerkin approximation in Sections
 \ref{sec:estimates}--\ref{sec:limit}. Additional estimates  for  the
regularized problem are obtained in Section \ref{sec:estimates2} and
we eventually remove the regularization by a limit passage in Sections
 \ref{sec:limit2}--\ref{sec:limit3}.

\section{Main result}
\label{sec:main1}

We devote this section to the statement of the main well-posedness result, namely
Theorem \ref{thm:main}.
Let us start by introducing some notation. 

\subsection{Functional analytic framework} 
\label{sec:main11}

Recall that we ask for
\begin{description}
 \item[{\rm (A0)}] Let $\Omega\subset \R^n$ be a nonempty bounded Lipschitz domain, and let $T\in (0,\infty)$.
\end{description}
In the following, we set $ Q:=\Omega\times (0,T) $ and we use the spaces
$$H:=L^2(\Omega), \quad V:=H^1_0(\Omega), \quad V^*:=H^{-1}(\Omega)$$
where $V^*$ is the dual of $V$, so that $(V,H,V^*)$ is a classical
Hilbert triplet. Note that the choice of the space $V$ corresponds to the case
of homogeneous Dirichlet boundary conditions. Other boundary
conditions can also be treated, by redefining $V$ to be a closed
subset of $H^1(\Omega)$ containing $H^1_0(\Omega)$. 
The symbol
$(\cdot,\cdot)$ stands for the standard scalar product in $H$ whereas
$\la\cdot,\cdot\ra$ is the duality pairing between $V^*$ and $V$. We
denote by $\|\cdot\|$ the norm in $H$, by
$\| \cdot \|_E$ the norm in the generic Banach space $E$, and 
by $c_V>0$ the norm of the injection $V\subset H$, namely, $\| u\|
\leq c_V \| u \|_V$ for all $u \in V$.  We note that $c_V$ is given as $c_V = \lambda_1(\Omega)^{-1/2}$, where $\lambda_1(\Omega)$ stands for the principal eigenvalue of the Dirichlet Laplacian in $\Omega$.  We set $A: V
\to V^*$ to be the Riesz mapping given by
$$\la Au,v \ra:=\int_\Omega \nabla u(x) \cdot \nabla v(x)\, \de x
\quad \forall  u,  v \in V.$$
In particular,  it holds  $\| u \|_V^2=\la Au,u\ra$ for all $u\in V$.

\subsection{Maximal monotone operators}
\label{sec:main12}

Given any maximal monotone graph $\gamma: \R \to 2^\R$ with $0\in \gamma(0)$,  for all $ r  \in D(\gamma):=\{ r  \in \R \ : \ \gamma( r )\not =\emptyset\}$ we
define by $\gamma^\circ ( r )$ the point in
the closed interval $\gamma( r )$ which is closest to $0$. One    can
uniquely find  (see e.g.~\cite[Example~2.8.1, p.~43]{Brezis73}) a  proper, lower semicontinuous and convex function $\haz \gamma: D(\haz \gamma)\subset \R\to [0,\infty]$ such that $\gamma =
\partial \haz \gamma$, where $0=\haz \gamma (0)=\min \haz \gamma$ and $\partial$ denotes the subdifferential in the sense
of convex analysis~\cite{Brezis73}. The symbol $D(\haz \gamma)$
indicates the  \emph{effective domain}  of $\haz \gamma$, namely, $D(\haz \gamma):=\{r\in \R \: :
\: \haz \gamma(r) <\infty\}$ .  % whereas  $D(\gamma):=\{r \in  D(\hat\gamma) \::\:\gamma(r)\not = \emptyset\}$ is the domain of $\gamma$.
In particular,
we have that $\haz \gamma \geq 0$, as well as $\haz \gamma^*\geq 0$,
where $\haz \gamma^*$ is the Legendre--Fenchel conjugate of $\haz
\gamma$, namely $\haz \gamma^*(r)  := \sup_{s  \in \R} (rs-\haz \gamma(s))$ for all $ r \in
\R$. The Fenchel identity $\haz \gamma(s) + \haz \gamma^*(r) =
rs$ for all $r \in \gamma(s)$ or, equivalently, all $s\in
 \partial \haz \gamma^*(r)$ in particular ensures that 
 $\gamma^{-1} =  \partial \haz \gamma^*$.
%$\haz{\gamma^{-1}} = \haz \gamma^*$.

We indicate by
$\psi_\gamma:H\to [0,\infty]$ the convex, proper, and lower
semicontinuous functional given by
\begin{equation}\label{eq:psia0}
    \psi_{\gamma}(v):= \begin{cases} \displaystyle\int_{\Omega} \haz{\gamma}(v(x))
      \,\de x \quad &\text{if} \  \haz \gamma \circ v \in L^1(\Omega), \\[2mm]
     \infty \quad &\text{else}\end{cases}
 \end{equation}
 for $v \in H$.  Moreover, we set $\Psi_\gamma:L^2(0,T;H) \to [0,\infty]$ to be
 \begin{equation}\label{eq:psia}
    \Psi_{\gamma}(v):= \begin{cases} \displaystyle\int_0^T
      \psi_\gamma(v) \, \de t = \int_0^T\!\!\int_{\Omega} \haz{\gamma}(v(x,t))
      \,\de x\, \de t \quad &\text{if} \  \haz \gamma \circ v \in L^1(Q), \\[2mm]
     \infty \quad &\text{else}\end{cases}
 \end{equation}
 for $v \in L^2(0,T;H)$.  
 
Let us consider the subdifferential $\gamma_H:=\partial \psi_\gamma:H \to 2^{H}$ and recall that
 $v\in \gamma_H (u)$  if and only if  $u\in  D(\psi_\gamma)$ and 
\[ (v,w-u) \leq \psi_\gamma(w) - \psi_\gamma(u) \quad \forall w \in H. \] 
In fact, the subdifferential $\gamma_H$ is characterized 
(cf.~\cite[Proposition~2.16, p.~47]{Brezis73}) as,  for $u,v \in H$,   
$$v\in \gamma_H(u) \ \ \Leftrightarrow \ \ %v \in H \ \ \text{and} \ \
v(x)\in \gamma(u(x)) \ \ \text{for a.e.}\ x \in \Omega$$
with domain $ D(\gamma_H)=\{u \in H \: : \: \exists v \in H
\ \text{with} \ v(x)\in \gamma(u(x)) \ \text{for a.e.}\ x \in
\Omega\}$.  In the following, for any $u \in D(\gamma_H)$ we use
the symbol  $ \gamma_H^\circ(u)$ to indicate the element of minimal
norm in the nonempty, convex, and closed set $\gamma_H(u)$. 
One can also consider the subdifferential $\partial
\Psi_\gamma:L^2(0,T;H)\to 2^{L^2(0,T;H)}$ and check that
$D(\partial \Psi_\gamma) = \{u \in L^2(0,T;H) \ : \ \exists v \in
L^2(0,T;H) \ \text{with} \ v \in \gamma
(u) \ \text{a.e.~in} \ Q \}$ and that, for all $u
\in D(\Psi_\gamma)$ one has 
$\partial \Psi_\gamma (u) = \{v \in L^2(0,T;H): \ v \in \gamma
(u) \ \text{a.e.~in} \ Q\}$.

% By restricting $\psi_\gamma$ to $V$, we define
% $\gamma_{V,V^*}:= \partial \psi_\gamma|_V:V \to 2^{V^*}$. Recall that
% $v\in \gamma_{V,V^*} (u)$ iff $u \in
% D(\psi_\gamma|_V):=D(\psi_\gamma)\cap V$ and 
% \[ \la v,w-u \ra \leq \psi_\gamma(w) - \psi_\gamma(u) \quad \forall
%   w \in V. \]

\subsection{Convergence of maximal monotone operators and convolution}
\label{sec:main13}

 Given  a sequence of maximal monotone maps $(\gamma_h)_h$
for $h>0$ with $0\in
\gamma_h(0)$,  one  says that $\partial \Psi_{\gamma_h} \to \partial \Psi_\gamma$ in the {\it
  graph sense in $L^2(0,T;H)$} as $h \to 0$ if, for all $u,\, v \in L^2(0,T;H)$ with $v\in \partial\Psi_{\gamma}(u)$, one can
find two sequences $(u_h)_h,\, (v_h)_h$  in  $L^2(0,T;H)$ with $v_h \in
\partial \Psi_{\gamma_h}(u_h)$ such that $u_h\to u$ and $v_h\to v$ in
$ L^2(0,T;H)$~\cite[Definition~3.58, p.~360]{Attouch}. A straightforward extension of~\cite[Proposition~3.59, p.~361]{Attouch} guarantees that
\begin{align}
  &\partial\Psi_{\gamma_h}\to \partial \Psi_\gamma \ \text{in the graph
    sense in $L^2(0,T;H)$}, \nonumber\\[2mm]
  &  v_h \in \partial \Psi_{\gamma_h}(u_h), \
  u_h \weakto u \ \text{in} \ L^2(0,T;H),  \,  v_h\weakto v \ \text{in} \ L^2(0,T;H), \nonumber\\
  &   \text{and} \ \limsup_{h \to 0}\int_0^T(v_h,u_h)\, \de t \leq
    \int_0^T(v,u)\, \de t \nonumber\\
  &\quad \Rightarrow \ v \in \partial \Psi_\gamma(u)
    \ \text{and}  \ \int_0^T(v_h,u_h)\, \de t\to \int_0^T(v,u)\, \de t.\label{eq:attouch}
\end{align}
Under the assumption  $0\in
\gamma_h(0)$, the graph convergence $\partial \Psi_{\gamma_h} \to \partial \Psi_{\gamma}$ is
equivalent to the {\it Mosco convergence} $\Psi_{\gamma_h}\to \Psi_\gamma
$  in $L^2(0,T;H)$  of the corresponding functionals~\cite[Theorem~3.66, p.~373]{Attouch},
namely, to the following  two  conditions 
\begin{align}
&\Psi_{\gamma}(u) \leq \liminf_{h\to 0} \Psi_{\gamma_h}
  (u_h)\quad \forall u_h \weakto u \ \text{in} \  L^2(0,T;H), \
                 \text{and}\label{eq:liminf}\\
  &\forall y \in L^2(0,T;H) \ \exists y_h \to y\ \text{in} \  L^2(0,T;H) \
    \text{such that} \ \Psi_{\gamma_h}
  (y_h)\to \Psi_{\gamma}(y).\label{eq:limsup}
\end{align}

Given any $a\in L^p(0,T)$ and $b \in L^q(0,T)$ (or $b \in L^q(0,T;E)$
with $E$ being a Banach space) with $1/p+1/q=1+1/r$ for some  $1 \le p,q,r \le \infty $, in the following, we indicate the standard convolution product on $(0,t)$ as
$$(a\ast b)(t)  :  = \int_0^t a(t-s)\,b(s)\, \de s$$
and we make use of Young's convolution inequality
\begin{equation}\label{eq:young}
\| a \ast b \|_{L^r(0,t)} \le \|a\|_{L^p(0,t)} \,\|b\|_{L^q(0,t)}
\quad \forall t \in [0,T].
\end{equation}

%\subsection{Assumptions}\label{sec:assumptions}

\subsection{Setting of the problem}
\label{sec:main14}

Our setting is specified as follows:
\begin{description}\setlength{\itemsep}{3mm}
 \item[{\rm (A1)}]  $\ell, \, \kappa: (0,T) \to \R$, $\ell, \, \kappa \in L^1(0,T)$  are nonnegative and nonincreasing such that $\ell \ast \kappa \equiv 1$.
 \item[{\rm (A2)}] $\alpha =\partial \haz \alpha: \R \to 2^\R$ is maximal monotone with a strictly convex potential $\haz \alpha$.
 \item[{\rm (A3)}] $\beta:\R \to 2^\R$ is maximal monotone.
 \item[{\rm (A4)}] $g:\Omega\times(0,T)\times \R \to \R$ is a
   Carath\'eodory function, and moreover, there exist  $ q  >2$ and
    $\Lambda_g > 0$  such that
\begin{align*}
& \quad |g(x,t,u_1)  - g(x,t,u_2)|\leq \Lambda_g|u_1-u_2| \quad \forall u_1,\, u_2\in \R, \ \text{a.e.} \ (x,t)\in Q,\\
  & \quad  |g(x,t,u)| \leq \Lambda_g(1+u^{2/ q })  \quad \forall u \in \R, \
    \text{a.e.} \ (x,t)\in Q.%\\
%  & \mbox{[IT FOLLOWS FROM THE ABOVE WITH $u=0$]} \quad g(\cdot,u_*(\cdot))\in L^2(Q) \ \text{ for some } \ u_*\in L^2(Q). 
\end{align*}
 \item[{\rm (A5)}] $v_0 \in \alpha_H(u_0)\cap  L^{2 q -2}(\Omega)$
    for some $u_0 \in D(\alpha_H) \cap D(\beta_H)$  with
   $\beta_H^\circ (u_0) \in L^{2 q -2}(\Omega)$,  where $q > 2$ is given in (\ref{ass_lip}).
\end{description}

{
\color{white}
\begin{align*}
\tag{A0}\label{Azero}\\[-10mm]
\tag{A1}\label{ciao}\\[-10mm]
\tag{A2}\label{ass2}\\[-10mm]
\tag{A3}\label{ass3}\\[-10mm]
\tag{A4}\label{ass_lip}\\[-10mm]
\tag{A5}\label{eq:init}
\end{align*}
}

Assumption \eqref{ciao} %in particular guarantees that the kernels
                        %$\kappa$ and $\ell$ are {\it completely
                        %positive}. This
covers the case of the
Riemann--Liouville fractional derivative of order
$\theta \in (0,1)$, corresponding to the
choices
$$
\kappa (t) = \frac{t^{-\theta}}{\Gamma(1{-}\theta)} \ \ \text{and}
\ \ \ell(t) = \frac{t^{\theta-1}}{\Gamma(\theta)} \quad \forall t \in (0,T).
$$
 Then $\kappa$ is also convex. 

Note that no growth conditions are imposed on $\alpha$ and
$\beta$ in \eqref{ass2}--\eqref{ass3}. Correspondingly, both maximal monotone operators $\alpha_H$
and $\beta_H$ are possibly
unbounded. Moreover, $\beta_H$ can be degenerate.
 
Assumption \eqref{eq:init} entails that $ D(\alpha_H)\cap
D(\beta_H)\not= \emptyset$.  Without loss of generality, in the
following, we additionally assume that
$$0\in \alpha (0) \ \ \text{and}  \ \ 0\in \beta (0).$$
 Indeed, this  can be achieved by considering the
shifted graphs
$\tilde \alpha(r) = \alpha(r  +  r^0) -  \alpha^0$ and $\tilde \beta(r) = \beta(r  +  r_0) - \beta^0$ for some $r^0\in D(\alpha)\cap
D(\beta)$, $\alpha^0 \in \alpha(r^0)$, and $\beta^0 \in \beta(r^0)$,
as well as the shifted Lipschitz function $\tilde g(x,t,r) = g(x,t,r + r_0)  -\beta^0$.

Given assumption \eqref{ass_lip} we can
define a mapping $G:L^2(0,T;H) \to L^2(0,T;H)$ setting $G(u)(x,t)  : =
g(x,t,u(x,t))$ for all $u\in L^2(0,T;H) $ and a.e.~$(x,t)\in
Q$. Indeed, for all such $u$ the
function $g(x,t,u(x,t))$ is measurable in $Q$ and
$$|g(x,t,u(x,t))| \leq |g(x,t, 0)| + \Lambda_g|u(x,t)| \leq  \Lambda_g(1+|u(x,t)|) \quad \text{for a.e.} \
(x,t)\in Q$$ 
where we also used the sublinearity bound in (A4) to control $
|g(x,t, 0)| \leq \Lambda_g$. 
This  ensures  that 
$G(u)\in L^2(0,T;H)$, as well as
\begin{equation}
  \label{eq:Gbound}
  \|G(u)(t)\|\leq  \Lambda_g(|\Omega|^{1/2}+\|u(t)\|)   \quad \forall u \in
  L^2(0,T;H)\quad \text{for a.e.} \ t \in (0,T).
\end{equation}
Moreover, for all $u_1,\, u_2 \in L^2(0,T,H) $ we have that 
\begin{equation}
  \label{ass_lip0}\| G(u_1)(t)- G(u_2)(t)\|\leq \Lambda_g \| u_1(t)-u_2(t)\| \quad \text{for a.e.} \ t\in (0,T)
  \end{equation}
and the sublinearity assumption in \eqref{ass_lip} ensures, with an
application of Young's inequality, that 
\begin{equation}
  \label{ass_lip2}
  (G(u),u) \leq \frac{1}{2c_V^2}\|u\|^2 + C_G \quad
  \forall u \in L^2(0,T;H), \ \text{a.e. in} \ (0,T)
\end{equation} 
 for $C_G>0$ given by
$$
C_G :=   c_V^2\Lambda_g^2|\Omega| +
  \frac{ q -2}{2 q} \Lambda_g^{\frac{ 2q }{ q -2}}\left(\frac{ q }{2(2+ q )c_V^2}\right)^{-\frac{ q + 2}{ q - 2 }}|\Omega|.
$$

%for $C_G:=\Lambda_g |\Omega|  + (2-\delta)(\delta c_V^2)^{\delta/(2-\delta)}\Lambda_g^{2/(2-\delta)} |\Omega|  /2$.

As  $(x,t)\mapsto g(x,t,0)\in L^\infty(Q)$,  one can find a  measurable 
set $J\subset  (0,T)$  of full measure, i.e., $|(0,T) \setminus J| =
0$,  such that %$x\mapsto u_*(x,t)\in H$ and
$x \mapsto
g(x,t, 0 )\in  L^\infty(\Omega)$  for all $t \in J$.
We may now define $\mathfrak{g}: (0,T)  \times H \to H$ by $\mathfrak{g}(t,v)(x) = g(x,t,v(x))$ for $t\in J$ and
$\mathfrak{g}(t,v)=0$ for $t \in  (0,T)  \setminus J$  for $v \in H$.  Moreover,  since $(x,t)\mapsto g(x,t,v(x))\in L^2(Q)$, one  observes  that $x\mapsto g(x,t,v(x))$ is measurable in $\Omega$, for all
$t\in (0,T)$. Moreover,
$$|\mathfrak{g}(t,v)(x)|\leq |g(x,t, 0 )| + \Lambda_g|v(x)| 
\quad \text{for a.e.} \ x \in \Omega, \ \forall t \in  (0,T)$$
which implies that $\mathfrak{g}(t,v)\in H$ for all $t\in (0,T)$ and that the map
$t\in (0,T)\mapsto \mathfrak{g}(t,v)\in H$ is strongly
measurable.  Eventually, 
\begin{equation}
  \label{eq:ass_lip3}
\|\mathfrak{g}(t,v_1) - \mathfrak{g}(t,v_2)\| \leq \Lambda_g \| v_1  - v_2 \|
\quad \forall v_1,\, v_2\in H, \ \forall t\in (0,T).
\end{equation}

 Let $X$ be a Banach space and let $1 \leq p \leq \infty$. Define a map $\mathcal B : D(\mathcal B) \subset L^p(0,T;X) \to L^p(0,T;X)$ by $u \mapsto \mathcal{B}(u) := \partial_t (\kappa*u)$ for $u \in D(\mathcal B) := \{u \in L^p(0,T;X) \colon \kappa*u \in W^{1,p}(0,T;X), \  (\kappa*u)  (0) = 0\}$. Then $\mathcal B$ is linear and $m$-accretive in $L^p(0,T;X)$ (see, e.g.,~\cite{Za09,Akagi19}). 
In the following, we use the nonlocal-in-time chain-rule inequality from
\cite[Proposition~4.2.~ii]{Akagi24}:  if $\gamma :\R \to 2^\R$ is a maximal monotone map,  for all $u_0
\in D(\psi_\gamma)$, 
$u \in L^2(0,T,H)$ with $\psi_\gamma(u)\in L^1(0,T)$ and $\kappa \ast
(u-u_0)\in W^{1,2}(0,T;H)$  satisfying $(\kappa * (u- u_0))(0) = 0$,  and
$z\in L^2(0,T,H)$ with $z \in \gamma_H(u)$ a.e.~in $(0,T)$, one has
that
\begin{equation}
  \label{eq:chain}
  \Big[ \ell\ast \big(\partial_t (\kappa \ast (u - u_0)), z\big)\Big](t) \geq 
  \psi_\gamma(u(t)) - \psi_\gamma(u_0) \quad \text{for a.e.} \ t \in (0,T).
\end{equation}
If $ \gamma$ is  (single-valued and) Lipschitz continuous, a nonlocal
chain-rule inequality holds also  for  $u\in L^2(0,T;V)$  satisfying  $\kappa \ast (u-u_0)\in W^{1,2}(0,T;V^*)$  and $(\kappa * (u- u_0))(0) = 0$  only.  Then we see $\gamma(u) \in L^2(0,T;V)$.  We argue by approximation: Let $u_\delta \in
W^{1,2}(0,T;V)$ be such that $u_\delta \to u$ in $L^2(0,T;V)$ and $\kappa \ast
(u_\delta-u_0)\to \kappa \ast
(u-u_0)$ in $W^{1,2}(0,T;V^*)$ as $\delta\to 0$.  Moreover, we find that $(\kappa *(u_\delta -u_0))(0) = 0$.  By passing to the
$\liminf$ as $\delta \to 0$ in inequality
\eqref{eq:chain} written for $u_\delta$ and $z_\delta =
\gamma_H(u_\delta)$ a.e.~in $(0,T)$ we get 
\begin{equation}
  \label{eq:chain2}
   \Big[\ell \ast \big\langle \partial_t (\kappa \ast (u - u_0)), z \big\rangle \Big](t) \geq 
  \psi_\gamma(u(t)) - \psi_\gamma(u_0) \quad \text{for a.e.} \ t \in (0,T).
\end{equation}

We are now in the position of stating our main result.

\begin{theorem}[Well-posedness]\label{thm:main}   Under 
  assumptions \eqref{Azero}--\eqref{eq:init} there exists a triplet
$(u,v,w)\in  L^2(0,T;V\times H \times H)$ 
  such that  $\kappa \ast (v - v_0) \in W^{1,2}(0,T; V^*)$, $(\kappa *(v - v_0))(0) = 0$  and 
  \begin{align}
    &\partial_t (\kappa \ast (v - v_0) )
       + A u + w =G(u) \quad \text{in} 
      \ V^*, \, \text{a.e.~in} \ (0,T), 
    \label{eq:prob5}\\
    & v \in \alpha(u) \quad \text{a.e.~in} \ Q, 
    \label{eq:prob2}\\
    &w  \in  \beta(u)\quad \text{a.e.~in}   \  Q.
    \label{eq:prob4}
  \end{align}
In addition, if $ \kappa$ is convex,  $\alpha$ is strongly monotone and $\beta$ is {\rm (}single-valued and{\rm )} Lipschitz continuous, such $u$ is unique.
\end{theorem}

The remainder of the paper is devoted to  giving a  proof of Theorem
 \ref{thm:main}. At first, we check uniqueness in Section
 \ref{sec:uniqueness}. Existence is then proved in
Section \ref{sec:existence} via an approximation and Galerkin discretization approach.

\section{Proof of Theorem \ref{thm:main}: uniqueness}\label{sec:uniqueness}

For the purpose of proving uniqueness,  $\kappa$ is assumed to be convex and  $\alpha$ and $\beta$ are assumed to be strongly monotone
and (single-valued and) Lipschitz continuous, respectively.  That is,  let
$C_\alpha>0$ and $\Lambda_\beta >0$ be such that  
\begin{align*}
C_\alpha|r_1-r_2|^2 &\leq (y_1-y_2)(r_1-r_2)\quad \forall
    y_i \in \alpha(r_i), \  i=1,2,\\
|z_1 - z_2| &\leq \Lambda_\beta|r_1-r_2|\quad \forall
    z_i  =  \beta(r_i), \  i=1,2.
  \end{align*}

As $\ell\in L^1(0,T)$, we can find $m\in {\mathbb N}$ large enough so
that  the number $\tau := T/m$ satisfies
  \begin{equation}
    \label{eq:L1}
    \| \ell\|_{L^1(0,\tau)}\leq
    \frac{C_\alpha}{\sqrt{2}(\Lambda_\beta + \Lambda_g)}.
  \end{equation}

Let $(u_i, v_i, w_i)$ for $i=1,2$ solve
\eqref{eq:prob5}--\eqref{eq:prob4} in the sense of Theorem \ref{thm:main} and define $\tilde{u}:=u_1-u_2$,
$\tilde{v}:=v_1-v_2$,
$\tilde{w}:=w_1- w_2$, and $\tilde G:=G(u_1)-G(u_2)$. 
Writing
\eqref{eq:prob5} for $(u_i, v_i, w_i)$, for $i=1,2$, in place of $(u, v, w)$ and
taking the difference of the two resulting equations we get  
    \[ \partial_t (\kappa\ast \tilde{v})+ A\tilde{u}   +
      \tilde{w}=\tilde G \quad \text{in} \  V^*, \ 
      \text{a.e.~in} \ (0,T).\]
Convolving with $\ell$ and using that $\ell\ast \kappa
     \equiv 1$, see \eqref{ciao}, one obtains
\begin{equation} \tilde{v}  +\ell \ast A\tilde{u} + \ell \ast
  \tilde{w}=\ell\ast \tilde G \quad  \text{in} \  V^*, \ 
  \text{a.e.~in} \ (0,T).\label{eq:usala}
\end{equation}

Test \eqref{eq:usala} by $\tilde{u}$ and integrate on $ (0,t)$ for $t \in
[0, T ]$. Using the strong monotonicity of $\alpha$ 
we get
\begin{align} &C_\alpha \int_0^t \|\tilde{u}\|^2 \,\de s  \le \int_0^t ( \tilde{v},\tilde{u}) \,\de s=\int_0^t \la \tilde{v},\tilde{u}\ra\, \de s \nonumber\\
& \quad \stackrel{\eqref{eq:usala}}{=}  -\int_0^t \la\ell \ast      A\tilde{u},  \tilde{u}\ra\,\de s-\int_0^t (\ell\ast\tilde{w},\tilde{u})\,\de s+\int_0^t(\ell\ast\tilde  G,\tilde  u)\,      \de s.\label{eq:rhs}\end{align}
  Since, in addition to (A0), we assumed that $ \kappa$ is convex, then,  as $\ell$ is  of positive type (see~\cite[p.~500, Proposition 3.1]{GLS}),  the first term in the above
right-hand side is  nonnegative,  namely,
$$\int_0^t\la \ell \ast A\tilde u, \tilde u\ra \, \de s =
\int_0^t\!\!\int_\Omega (\ell\ast \nabla \tilde u)\cdot \nabla \tilde
u \,\, \de x \,\de s \geq 0.$$ 
The second and third terms in the right-hand side of \eqref{eq:rhs}
can be controlled as follows
\begin{align*}
&  -\int_0^t            (\ell  \ast  \tilde{w},\tilde{u})\, \de  s +      \int_0^t(\ell\ast\tilde G,\tilde  u)\, \de   s\le \int_0^t \|\ell \ast \tilde{w}\|\|\tilde{u}\|
        \,\de s +  \int_0^t \|\ell\ast\tilde G \|\|\tilde{u}\|
        \,\de s \\
& \quad\le \int_0^t (\ell \ast \|\tilde{w}\|) \|\tilde{u}\| \,\de
                     s+\int_0^t (\ell \ast \|\tilde{G}\|)
                     \|\tilde{u}\| \,\de s \le (\Lambda_\beta
                     +\Lambda_g)\int_0^t (\ell \ast \|\tilde{u}\|)
                     \|\tilde{u}\| \,\de s \\
  &\quad \leq \frac{C_\alpha}{2}\int_0^t\| \tilde u\|^2\, \de s +
    \frac{(\Lambda_\beta+\Lambda_g)^2}{2C_\alpha}\int_0^t (\ell \ast\|
     \tilde u \|)^2 \, \de s,\nonumber
\end{align*}
 where we have used  the Lipschitz continuity of $\beta$ and $G$.  Therefore %the Young convolutioninequality \eqref{eq:young}  and 
\eqref{eq:rhs} reads, 
\begin{equation}\label{eq:general}
\frac{C_\alpha}{2}\int_{0}^t \| \tilde u \|^2 \, \de s \leq \frac{(\Lambda_\beta+\Lambda_g)^2}{2C_\alpha}\int_0^t (\ell \ast\|
     \tilde u \|)^2 \, \de s.
\end{equation}
Now, if $t=\tau$, we can use  the Young convolution inequality \eqref{eq:young} and inequality \eqref{eq:L1}  in order to  obtain 
\[ \frac{C_\alpha}{2}\int_{0}^t \| \tilde u \|^2 \, \de s \leq 
    \frac{(\Lambda_\beta+\Lambda_g)^2}{2C_\alpha}\|\ell\|_{L^1(0,\tau)}^2 \int_0^t
                                                                                                                                       \|\tilde{u}\|^2
                                                                                                                                       \de
                                                                                                                                       s
    \stackrel{\eqref{eq:L1}}{\leq}\frac{ C_\alpha}{4}\int_0^t
                                                                                                                                       \|\tilde{u}\|^2
                                                                                                                                       \de
                                                                                                                                       s
                                                                                                                                       .
\]
We deduce that $\tilde u=0$ a.e.~in $\Omega \times
(0,\tau)$.

We now prove that, if $\tilde u=0$  a.e.~in $\Omega \times
(0,j\tau)$ for some $j=1,\dots,m-1$, one also has that $\tilde u=0$  a.e.~in $\Omega \times
(0,(j+1)\tau)$. This will imply that $\tilde u=0$  a.e.~in $Q$.  From \eqref{eq:general} with $t \in [j \tau, (j+1)\tau]$ and
 the fact that $\tilde u=0$  a.e.~in $\Omega \times
(0,j\tau)$ we get that
\begin{equation}
  \frac{C_\alpha}{2}\int_{j\tau}^t \| \tilde u \|^2 \, \de s \leq
\frac{(\Lambda_\beta + \Lambda_g)^2}{2C_\alpha}\int_{j\tau}^t (\ell
\ast \| \tilde u \|)^2 \, \de s.\label{eq:uni}
\end{equation}
Let $\tilde \ell_{ t } :(0, t )\to {\mathbb R}$ be defined as
$$\tilde \ell_{ t } (\sigma):=
\left\{
  \begin{array}{ll}
    \ell(\sigma)\quad &\text{for} \ \  0 < \sigma \leq t - j\tau,\\
                   0& \text{for} \ \  t - j\tau  < \sigma < t.
  \end{array}
\right.$$
Note that, for  $ j \tau < r < s \le t $  we have
$ 0 < s-r<t-j\tau\leq \tau$. This  in  particular entails that
\begin{equation}
  \ell(s-r)\|
      \tilde u(r)\|=\tilde \ell_{ t }(s-r)\|
      \tilde u(r)\| \quad \text{for} \     j \tau < r < s \le t    .\label{eq:s}
      \end{equation}
 One can use this fact in order to control the
  right-hand side of \eqref{eq:uni} as follows:
  \begin{align*}
\MoveEqLeft{
    \frac{(\Lambda_\beta + \Lambda_g)^2}{2C_\alpha}\int_{j\tau}^t (\ell
\ast \| u \|)^2 \, \de s  = \frac{(\Lambda_\beta + \Lambda_g)^2}{2C_\alpha}\int_{j\tau}^t\left(\int^s_{0 } \ell(s-r)\|
      \tilde u(r)\|\, \de r \right)^2 \de s
}\\
    &= \frac{(\Lambda_\beta + \Lambda_g)^2}{2C_\alpha}\int_{j\tau}^t\left(\int^s_{ j\tau } \ell(s-r)\|
      \tilde u(r)\|\, \de r \right)^2 \de s\\
    &\hspace{-1mm} \stackrel{\eqref{eq:s}}{=} \frac{(\Lambda_\beta + \Lambda_g)^2}{2C_\alpha}\int_{j\tau}^t\left(\int^s_{ j\tau }\tilde \ell_{ t }(s-r)\|
      \tilde u(r)\|\, \de r \right)^2\de s\\
    &= \frac{(\Lambda_\beta + \Lambda_g)^2}{2C_\alpha}\int_{0}^t\left(\int_0^s\tilde \ell_{ t }(s-r)\|
      \tilde u(r)\|\, \de r \right)^2\de s\nonumber  = \frac{(\Lambda_\beta + \Lambda_g)^2}{2C_\alpha} \left\| \tilde \ell_{t} * \|\tilde{u}\| \right\|_{L^2(0,t)}^2 \nonumber\\
    &\hspace{-1.4mm} \stackrel{\eqref{eq:young}}{\leq}\frac{(\Lambda_\beta + \Lambda_g)^2}{2C_\alpha}\|\tilde \ell_{ t }\|_{L^1(0,t)}^2\int_0^t \|
      \tilde u \|^2\, \de s= \frac{(\Lambda_\beta + \Lambda_g)^2}{2C_\alpha}\|
      \ell\|^2_{L^1(0,t-j\tau)}\int_{j\tau}^t \|
      \tilde u \|^2\, \de s \\
    &\leq \frac{(\Lambda_\beta + \Lambda_g)^2}{2C_\alpha}\|
      \ell\|^2_{L^1(0,\tau)}\int_{j\tau}^t \|
      \tilde u \|^2\, \de s \stackrel{\eqref{eq:L1}}{\leq} \frac{C_\alpha}4 \int_{j\tau}^t \|
      \tilde u \|^2\, \de s
  \end{align*}
 for $t \in [j\tau, (j+1)\tau]$.  In combination with \eqref{eq:uni}, this proves that $\tilde u=0$ a.e.~in $\Omega\times
  (j\tau,t)$ for all $t\in[ j\tau,(j+1)\tau ]$. We have hence checked
  that  $\tilde u=0$ a.e.~in $\Omega\times
  (0,(j+1)\tau)$ and the uniqueness assertion follows.

\section{Proof of Theorem \ref{thm:main}: existence}\label{sec:existence} 

As mentioned in Section \ref{sec:intro}, the proof of the existence of solutions to
\eqref{eq:prob5}--\eqref{eq:prob4} follows  from regularization  and
Galerkin approximation procedures.

For the reader's convenience, we
split the argument into subsections. 
At first, we regularize the maps
$\alpha$ and $\beta$ by passing to their Yosida approximations and by
adding a multiple of the identity  to $\alpha$  (Section
 \ref{sec:regularization}). The ensuing regularized problem is then
tackled via a Galerkin approximation (Section \ref{sec:galerkin}). After
establishing some a-priori estimates (Section \ref{sec:estimates}), we pass to the limit in the Galerkin
approximation (Section \ref{sec:limit}) and prove the well-posedness
of the regularized problem, namely, Proposition \ref{cor:reg}  below.  We then
obtain additional a-priori estimates  for  the solution to the
regularized problem (Section \ref{sec:estimates2}) and pass to the
limit in the regularizations (Sections
\ref{sec:limit2}--\ref{sec:limit3}), eventually proving Theorem \ref{thm:main}. 

\subsection{Regularization}\label{sec:regularization}

Given a parameter $\epsi  \in (0,1)$, we denote by
$$
\alpha_\epsi:= \frac{{\rm id}_\R  -({\rm id}_\R + \epsi\alpha)^{-1}}{\epsi},
\quad  \tilde \beta_\epsi  := \frac{{\rm id}_\R -({\rm id}_\R  + \epsi \beta)^{-1}}{\epsi},
$$
the Yosida approximations of the maximal monotone graphs $\alpha$ and
$\beta$ at level $\epsi$, respectively. Here, ${\rm id}_\R $ stands for the identity in
$\R$.  Moreover, we define the truncated  one by 
$$\beta_\epsi ( r ) = \min\{ \max\{\tilde \beta_\epsi( r ), - \epsi^{-1}\}, \epsi^{-1}\} \quad \forall  r 
\in \R.$$ 
Recall that $\alpha_{\epsi H}$
and $\beta_{\epsi H}$ are single-valued, $\epsi^{-1}$-Lipschitz
continuous  in $H$,  and that
$\alpha_{\epsi H}(0)=\beta_{\epsi H}(0)=0$ 
(cf.~\cite[Proposition~2.6 (i), p.~28]{Brezis73}). % Moreover, one has that $\|
% \beta_{\epsi H}(u_0)\|\leq \| \beta_H(u_0)^\circ\| \leq C$, where we
% recall that $ \beta_H(u_0)^\circ$ indicates the element of minimal
% norm in the nonempty, convex, and closed set $\beta_{\epsi H}(u_0)$.
% Furthermore, moving from
% \eqref{ass2} one can prove that there exists $\epsi^*\in(0,1)$ and
% $C_\alpha>0$, both depending on $c_\alpha$, such that 
% \begin{equation}
%   \label{eq:betaco}
% \haz{\alpha_\epsi}(r)\geq \frac{c_{\alpha}}{2}|r|^2 -C_\alpha \quad \forall
%   r \in \R, \ \forall \epsi \in (0,\epsi^*).
% \end{equation}

We additionally define
$$\alpha_\nu:=\nu \,{\rm id}_{ \R} + \alpha, \quad  \alpha_\ne:= \nu \,{\rm id}_{ \R} + \alpha_\epsi, \quad \eta_\ne :=
\alpha_\ne^{-1}$$
for all $\nu \in (0,1)$.
Note that $\alpha_{\nu\epsi H}$ is  monotone,  $(\nu{+}\epsi^{-1})$-Lipschitz continuous, and
coercive from $H$ to
$H$. Hence,  it is maximal monotone and onto. The mapping $\eta_{\ne H}$ is
$\nu^{-1}$-Lipschitz continuous  in $H$.  Moreover,  we have  $\alpha_{\ne
  H}(0)=\eta_{\ne H}(0)=0$.

% Let us quickly show how \eqref{eq:a2} passes from $\beta$ to $\beta_\nu$: denoting $J_\nu:=(I+\nu \beta)^{-1}$, we have, for every $r \in \R$, that
% \begin{align*}
%  |\beta_\nu(r)|^2 &\le C_\beta \left [ 1+ (J_\nu r)(\beta_\nu(r)) \right ] = C_\beta \left [ 1+ (J_\nu r-r)(\beta_\nu(r)) + r\beta_\nu(r) \right ] \\
%  &= C_\beta ( 1- \nu|\beta_\nu (r)|^2 + r\beta_\nu(r)) \le C_\beta (1+r\beta_\nu(r))
% \end{align*}
% so that
% \begin{equation}\label{eq:a2tau}
% |\beta_\nu(r)|^2 \le  C_\beta (1+r\beta_\nu(r)) \quad \text{ for every } r \in \R.
% \end{equation}

% Similarly we can see that 
% \begin{equation}\label{eq:a2tauu}
% |h|^2 \le (C_\alpha+2\nu)(1+rh) \le (C_\alpha+2)(1+rh) =: C'_\alpha(1+rh) \quad \text{ for every } (r,h) \in \alpha_\nu.
% \end{equation}

% \medskip

Recall that $\alpha_\epsi(r) \to \alpha^\circ(r)$  and $|\alpha_\epsi(r)| \le |\alpha^\circ(r)|$ for all $r\in D(\alpha)$ and $\beta_\epsi(r)\to \beta^\circ(r)$ and $|\beta_\epsi(r)| \le |\beta^\circ(r)|$  for all $r\in D(\beta)$, where  we recall that  $\alpha^\circ(r)$ and $\beta^\circ(r)$  denote  the elements of smallest absolute
value in the
closed intervals $\alpha(r)$ and $\beta (r)$, respectively  (cf.~\cite[Proposition 2.6(iii)]{Brezis73}). Moreover,  we see that  $\haz{\alpha_\epsi} \nearrow \haz \alpha$, $\haz{\alpha_\ne} \nearrow \haz \alpha_\nu$, and
$\haz{\beta_\epsi} \nearrow \haz \beta$  pointwisely. 
Owing to~\cite[Theorem~3.20, p.~298]{Attouch} we have that $\Psi_{\alpha_\epsi}
\to \Psi_\alpha$, $\Psi_{\alpha_\ne}
\to \Psi_{\alpha_\nu}$, and $\Psi_{\beta_\epsi} \to \Psi_\beta$ in the Mosco
sense in $L^2(0,T;H)$. Hence,~\cite[Theorem~3.66, p.~373]{Attouch} implies that 
\begin{equation}
  \label{eq:attouch2}
  \partial \Psi_{\alpha_{\ne}} \to \partial \Psi_{\alpha_\nu} \quad
  \text{and}\quad \partial \Psi_{\beta_\epsi} \to \partial
  \Psi_{\beta}\quad \text{in the graph sense in}  \ L^2(0,T,H)
\end{equation}
 as $\epsi \to 0$.  Moreover,   one can easily check  that 
\begin{equation}
  \label{eq:attouch3}
  \partial \Psi_{\alpha_{\nu}} \to \partial \Psi_{\alpha}\quad  \text{in the graph sense in}  \ L^2(0,T,H)
\end{equation}
 as $\nu \to 0$. 

 We next introduce some approximation $(u_{0\epsi},v_{0\epsi})$ of
initial data $(u_0,v_0)$ satisfying (A5), for which we can check not
only  that  $\alpha_{\epsi H}(u_{0\epsi}) \to v_0$ in $H$ but
also  that  $\beta_{\epsi H}(u_{0\epsi})$ is bounded in $H$
as $\epsi \to 0$.  This approximation  will play a crucial role when $\alpha$ is not supposed to be single-valued (in  other words, when $\alpha$ is single-valued, i.e., $\alpha_H(u_0) = v_0$, the following construction is not necessary). 
Let $( s_j )_{j\in J}$  indicate  all values in  $D(\alpha)$ where  the value of  $\alpha$ is not a singleton. 
As $\alpha$ is monotone  in $\R$,  the index set $J$ is at most countable. Starting from $u_0 \in
D(\alpha_H)$  as in \eqref{eq:init}, for all $j\in J$ we define the measurable sets
$\Omega_j:=u_0^{-1}( s_j )$ and we set $ \hat{\Omega}:= \Omega
\setminus \cup_{j\in J}\Omega_j$. Note that these sets form a
disjoint partition of $\Omega$, up to null sets. We set
\begin{align}
  \label{eq:u0}
  &u_{0\epsi}:=
  \left\{
    \begin{array}{ll}
      u_0\quad&\text{on} \  \hat{\Omega},\\
       s_j +\epsi v_0&\text{on} \ \Omega_j, \ \forall j \in J ,\\
    \end{array}
  \right.\\
  & v_{0\epsi}:=\alpha_{\epsi H}(u_{0\epsi}),\label{eq:v0}
\end{align}
 where $v_0$ is as in \eqref{eq:init}.  
 We can now state several properties of the approximation of the initial data.

\begin{lemma}[Properties of the approximation of the initial data]\label{L:regu_data}
 Under assumptions \eqref{Azero}, \eqref{ass2}--\eqref{eq:init}, and positions \eqref{eq:u0}--\eqref{eq:v0}  we
have the following\/{\rm :}
\begin{enumerate}
 \item[\rm (i)] It holds that
\begin{equation}
  \label{eq:v0conv} v_{0\epsi}\to v_0 \ \ \text{in} \ H \ \ \text{and}
  \ \ \| v_{0\epsi}\| \leq \| v_0\|
\end{equation}
as $\epsi \to 0$.
 \item[\rm (ii)] It holds that
\begin{equation}\label{eq:u0conv}
   | u_{0\epsi} - u_0|\leq \epsi |v_0|\quad \text{a.e. in} \ \Omega, 
\end{equation}
which in particular yields
\begin{align}
 \| u_{0\epsi} - u_0\| &\leq \epsi \|v_0\|, \label{eq:u0conv2}\\
 |\beta_{\epsi H}(u_{0\epsi})| &\leq |\beta_{
   H}^\circ(u_0)| + |v_0|\quad \text{a.e. in} \ \Omega, \label{eq:beta0}
\end{align}
for any $\epsi  \in (0,1)$. 
\item[\rm (iii)] For all $\epsi \in (0,1)$ and $\nu \in (0,1)$, set
$$
v_{0\ne}:=\alpha_{\ne  H}(u_{0\epsi}) = \nu u_{0\epsi}+ v_{0\epsi},
$$
where $u_{0\epsi}$ and $v_{0\epsi}$ are defined in \eqref{eq:u0}--\eqref{eq:v0}.
Then, there exists a constant $C > 0$  depending on $\|u_0\|$ and
$\| v_0\|$ but independent of $\epsi$ and $\nu$  such that
    \begin{align*}
  \int_{\Omega}
  \haz{\eta_{\ne}}(v_{0\ne}) \,\de x
&\leq {\nu} \|u_{0\epsi}\|^2 +
        |(u_{0\epsi},v_{0\epsi})|\leq  C,
    \end{align*}
where $\haz{\eta_{\ne}} = \haz{\alpha_\ne}^*$.  This  in
particular implies  that  $v_{0\ne} \in D(\psi_{\eta_{\ne}})$, for any $\epsi,\nu \in (0,1)$.
%    where we also used \eqref{eq:v0conv}--\eqref{eq:u0conv}.
 \item[\rm (iv)]  Let $q > 2$ be given as  in   {\rm (\ref{ass_lip})}.  For all $\epsi \in (0,1)$, setting $\beta_{\epsi  q}( r ):= |\beta_\epsi( r )|^{ q -2} \beta_\epsi( r ) $ for all $ r  \in \R$, it  holds that
$$
\int_\Omega \haz{\beta_{\epsi  q}\circ
  \alpha_{\ne}^{-1}}(v_{0\ne})\, \de x \leq  \tilde C   (\nu \|u_{0  \epsi } \| + \|v_{0\epsi}\|),
$$ 
  where $\tilde C>0$ depends on $ q $,
$\|\beta^\circ_H(u_0)\|_{L^{2 q -2} (\Omega)}$, and $\|v_0\|_{L^{2 q -2} (\Omega)}$ but is independent of $\epsi$  and $\nu$.  This  also yields that $v_{0\ne} \in D(\psi_{\beta_{\epsi  q} \circ \alpha_{\ne}^{-1}})$.
\end{enumerate} 
\end{lemma}

\begin{proof}
We prove (i).  As $v_{0\epsi} = \alpha_\epsi(u_0)$ a.e.~on $ \hat{\Omega}$,   since there  $\alpha(u_0)$ is  singleton,  i.e., $\alpha(u_0) = \{v_0\} = \{\alpha^\circ(u_0)\}$,  one has that
$v_{0\epsi}\to \alpha^\circ(u_0)=v_0$  and $|v_{0\epsi}| \le  |\alpha^\circ(u_0)| =  |v_0|$ a.e.~on  $ \hat{\Omega}$. On the other hand, for all $j\in J$ one has
\begin{align*}
  &v_{0\epsi} = \alpha_\epsi ( s_j +\epsi v_0 ) = \frac{1}{\epsi}\left(
    ( s_j +\epsi v_0 )  - ({\rm id} + \epsi \alpha)^{-1}( s_j + \epsi v_0 
    )\right)=v_0 \quad \text{a.e.~on} \ \Omega_j,
\end{align*}
where we used that $ s_j +\epsi \alpha ( s_j ) \ni  s_j +\epsi v_0$,
 and hence, 
$({\rm id} + \epsi \alpha)^{-1}( s_j +\epsi v_0 )= s_j $. We have hence proved that $v_{0\epsi}\to v_0$ and
    $|v_{0\epsi}| \leq |v_0|$ a.e.~in
    $\Omega$, whence \eqref{eq:v0conv} follows by  the  dominated convergence.

 As for (ii),  as we have that
    $$ u_{0\epsi}- u_0:=
  \left\{
    \begin{array}{ll}
     0 \quad&\text{on} \  \hat{\Omega},\\
      \epsi v_0&\text{on} \ \Omega_j, \ \forall j \in J ,\\
    \end{array}
  \right.$$
   we readily  obtain \eqref{eq:u0conv}  and \eqref{eq:u0conv2}.  This in particular implies that
\begin{align}
 & | \beta_{\epsi H}(u_{0\epsi})|\leq |\beta_{\epsi H }(u_0)| + |
   \beta_{\epsi H }(u_{0\epsi}) - \beta_{\epsi  H}(u_{0})| \nonumber\\
  &\quad  \leq
  |\beta_H^\circ(u_0)| + \frac{1}{\epsi}|  u_{0\epsi} -
  u_0| \stackrel{\eqref{eq:u0conv}}{\leq} |\beta_H^\circ(u_0)| + | v_0 |
    \quad \text{a.e. in} \ \Omega.  \nonumber
\end{align}
 Thus \eqref{eq:beta0} follows.  

We next prove (iii). Since $v_{0{\ne}} = \alpha_{\ne H} (u_{0  \epsi  })$ we have
    \begin{align*}
  \int_{\Omega}
  \haz{\eta_{\ne}}(v_{0\ne}) \,\de x&=
         %\int_{\Omega}\widehat{\alpha_{\ne}^{-1}}(v_{0\ne}) \,\de x =
         \int_{\Omega}\haz{\alpha_{\ne}}^*(v_{0\ne })\,\de
         x = -\int_{\Omega}\haz{\alpha_{\ne}}(u_{0 \epsi})\,\de x + (u_{0 \epsi},
         v_{0\ne }) \\
      &\le |(u_{0\epsi}, v_{0\ne})| \leq {\nu} \|u_{0\epsi}\|^2 +
        |(u_{0\epsi},v_{0\epsi})|\\
      & \leq 2 \nu
    \|u_0\|^2 + 2 \nu \| u_{0\epsi}-u_0\|^2 +
      (\|u_0\|+\|u_{0\epsi}-u_0\|)\|v_{0}\|\\
      &\leq  2\nu \|u_0\|^2 +\|u_0\|\,\|v_0\|+ (2\nu  \epsi +1)\epsi \|v_0\|^2.
    \end{align*}
    where we also used \eqref{eq:v0conv}  and \eqref{eq:u0conv2}. 
    
 We finally prove (iv). Recalling  that %the relations $v_{0\ne}
                                %\in \beta_\epsi \circ
                                %\alpha_{\ne}^{-1}(v_{0\ne})$ and
  $u_{0\epsi} = \alpha_{\ne}^{-1}(v_{0\ne})$ a.e.~in $\Omega$, we see that
    \begin{align*}
      &\int_\Omega \haz{\beta_{\epsi  q}\circ
      \alpha_{\ne}^{-1}}(v_{0\ne})\, \de x \leq ((\beta_{\epsi  q
         H}\circ
     \alpha_{\ne H}^{-1})(v_{0\ne}), v_{0\ne}) =   (\beta_{\epsi  q  H}
      (u_{0\epsi}), v_{0\ne}) \\
      &\quad \leq  \|
    \beta_{\epsi  q  H}(u_{0\epsi})\| (\nu \|u_{0  \epsi } \| + \|v_{0\epsi}\|).
    \end{align*}
     The assertion follows by using \eqref{eq:beta0} in order to
    check that
    \begin{align*}
      &\| \beta_{\epsi  q H} (u_{0\epsi})\|  = \left(\int_\Omega
      |\beta_{\epsi H}(u_{0\epsi})|^{2 q -2}\, \de x\right)^{1/2}
      \stackrel{\eqref{eq:beta0}}{\leq}
      \left(\int_\Omega(|\beta^\circ_H(u_0)| + |v_0|)^{2 q -2}\, \de
      x\right)^{1/2}\\
      &\quad \leq  2^{ q -3/2} ( \| \beta^\circ_H(u_0)\|_{L^{2 q-2}(\Omega)}^{ q -1} + \| v_0
        \|_{L^{2 q -2}(\Omega)}^{ q -1}).
    \end{align*}
    This completes the proof.
\end{proof}

\subsection{Galerkin approximation}\label{sec:galerkin}

Let $(e_i)_i \subset V$ denote normalized eigenfunctions of the Laplacian with  the  homogeneous
Dirichlet boundary  condition,  i.e., solutions  $e_i \in V$  to $ A e_i = \lambda_i
e_i $ with $ \| e_i\|_V=1$
for every $i \in \N$ and increasing eigenvalues $\lambda_i >0$.
Such a set is orthogonal in $H$ and orthonormal in $V$. By setting 
\[ H_n:= \Span\{e_1, \dots, e_n\}  \simeq \R^n\]
for all $n\in{\mathbb N}$ one has that the  closures  of $ \cup_n H_n$ in $H$ and $V$  coincide  with $H$ and $V$, respectively.
Let   $\pi_n:H\to H_n$ denote the orthogonal projection on $H_n$ and recall that
$ \pi_n(u)=\sum_{i=1}^n u_i e_i$ for all $u \in H$
where $ u_i:= (u,e_i)$ for $i=1,
  \dots, n$.
We also introduce the mapping
${r}_n: \R^n \to H_n$ given by
\begin{align*}
{r}_n(\xx):= \sum_{i=1}^n x_i e_i \quad \forall \xx =(x_1,\dots,x_n)\in
  \R^n.
\end{align*}

 We claim  that 
\begin{align}
  (\pi_n \circ \alpha_{\ne H})|_{H_n} &=(\pi_n \circ(\nu \,{\rm id}_H
                                        +\alpha_{\epsi  H}))|_{H_n}\nonumber\\
  &= \nu \, {\rm id}_{H_n} + (\pi_n
  \circ \alpha_{\epsi H})|_{H_n}:H_n \to H_n\ \ \text{is onto}\label{eq:vecchiolemma}
\end{align}
where ${\rm id}_H$ and ${\rm id}_{H_n}$  denote  the identities 
in $H$ and $H_n$, respectively.
 Indeed, \eqref{eq:vecchiolemma} follows from the fact that  $ \nu \, {\rm id}_{H_n}+ (\pi_n \circ
 \alpha_{\epsi H})|_{H_n}$ is  well  defined on $H_n$, strongly
  monotone and Lipschitz continuous  on $H_n$, hence maximal
 monotone and onto.  In particular, its inverse
 $$\eta_{\ne n H}: =  ((\pi_n \circ \alpha_{\ne H})|_{H_n})^{-1} =  (\nu \, {\rm id}_{H_n}+ (\pi_n \circ
 \alpha_{\epsi H} )|_{H_n} )^{-1} :H_n \to H_n$$ is well defined, monotone, and Lipschitz continuous  (and therefore, maximal monotone).

We now proceed  to  the Galerkin approximation. To start with, for all
 $\epsi,\,  \nu \in (0,1)$,  and $n \in \N$ we set
$$
z_{0\ne n}:=\pi_n(v_{0\ne}) = (\pi_n
\circ \alpha_{\ne H})(u_{0\epsi}).
$$
 Here $v_{0\ne}$ is given by Lemma \ref{L:regu_data}. 
We assume $\zz_{0\ne n}\in \R^n$ to  be  given by
$(\zz_{0\ne n})_i:=(z_{0\ne n},e_i)$, for $i=1,\dots,n$, and we define the mapping 
$\ddelta_{\ne n}: [0,T]\times \R^n \to \R^n$ as
\begin{align*}
(\ddelta_{\ne n}(t,\xx))_i&:=  \lambda_i((\eta_{\ne n  H }\circ
  {r}_n)(\xx),e_i) + ((\beta_\epsi \circ \eta_{\ne n  H }\circ
                            {r}_n)(\xx), e_i) \\
  &\quad- ((\mathfrak{g}(t,\cdot) \circ \eta_{\ne n  H }\circ
                 {r}_n)(\xx), e_i) \quad \text{for} \  i=1, \dots, n.
\end{align*}                 
Notice that
 $\ddelta_{\ne n}(t,\cdot)$ is Lipschitz continuous  in $\R^n$  uniformly  for  $t\in [0,T]$, as it is a composition of
 Lipschitz continuous maps  (see in particular \eqref{eq:ass_lip3})  and $H_n \simeq \R^n$. 
An application of~\cite[Theorem~5.1]{Akagi19} then ensures that there exists
a unique $\zz_{\ne n} \in L^2(0,T; \R^n)$ such that  $\kappa\ast (\zz_{\ne n}-\zz_{ 0\ne n}) \in W^{1,2}(0,T; \R^n)$, $(\kappa\ast (\zz_{\ne n}-\zz_{ 0\ne n}))(0) = 0$, and  
\begin{equation}\label{eq:vector}
    \partial_t (\kappa\ast (\zz_{\ne n}-\zz_{ 0\ne n}))(t) +\ddelta_{\ne
      n}(t,\zz_{\ne n} (t))=0 \quad \text{in } \R^n, \ \text{for a.e.} \  t \in (0,T).
  \end{equation}
 Indeed, in the notation of~\cite[Theorem~5.1]{Akagi19} it suffices to
 choose $H:=\R^n$, $\varphi:=0$, $f:=0$,  and
 $ F:=\ddelta_{\ne n}$.

 Set now $z_{\ne n}:={r}_n(\zz_{\ne n})  \in  L^2(0,T;H_n)$ and
 $u_{\ne n} :=  \eta_{\ne n H}(z_{\ne n})  \in  L^2(0,T;H_n)$, so that
 $$z_{\ne n}= (\pi_n \circ
    \alpha_{\ne H})(u_{\ne n}).$$
By multiplying the $i$-th component of \eqref{eq:vector} by
$e_i$ and summing from $ i= 1$ to $ i = n$ we find that  
\begin{align}
  &   \partial_t( \kappa\ast ( z_{\ne n}
      -z_{ 0 \ne n})) + A u_{\ne n} + (\pi_n \circ
  \beta_{\epsi H})(u_{\ne n}) \nonumber\\
  &\quad= (\pi_n\circ G)(u_{\ne n}) \quad \text{in} \   H_n,  \ \text{a.e.~in}  \  (0,T), \label{eq:probapprn}
    \end{align}
 where $A u_{\nu \epsi n}$ also denotes its realization in $H$, i.e., $A u_{\nu \epsi n} = \sum_{i=1}^n \lambda_i (u_{\nu \epsi n}, e_i) e_i \in H_n$. 
In the following, we shall use the
notation 
\begin{equation}\label{v_nuen}
v_{\ne n}:=\alpha_{\ne H}(u_{\ne n})  \in  L^2(0,T;H),% \quad z_{0{\ne
% n}}:=z_{ 0\ne n},
\end{equation}
so that $z_{\ne n} =\pi_n(v_{\ne n})  \in  L^2(0,T;H_n)$. % and $u_{\ne n} = z_{\ne  n   H}(v_{\ne n})$.

\subsection{A-priori estimates  for  the Galerkin approximation}\label{sec:estimates} 

For the sake of notational simplicity, henceforth we use the same
symbol $C$ to denote a generic positive constant, possibly depending on data
but independent of $\epsi$, $\nu$, and $n$. The actual value of the constant $C$
may vary from line to line. We will denote by $C(\epsi)$ a generic
positive constant depending on data and on~$\epsi$, but not on  neither  $\nu$
nor $n$.

% \begin{lemma}\label{le:est1} Under Assumption \ref{ass:leass}, there exists a sequence $\nu_n \downarrow 0$ and a constant $C>0$ such that
% \[ \|v_{\nu_n}^n\|_{L^\infty(0,T; H)} + \|\partial_t (( z^n_{\nu_n}-\pi_n(\alpha_{0{\nu_n}})\ast \kappa)\|_{L^2(0,T; H^{-1}(\Omega))} +\|u^n_{\nu_n}\|_{L^2(0,T; H^1_0(\Omega))} + \|\beta_{\nu_n}(u^n_{\nu_n})\|_{L^2(0,T; H)} \le C \]
% for every $n \in \N$. In particular, also $\|z_{\nu_n}^n\|_{L^\infty(0,T; H)}  \le C$.
% \end{lemma}
%\begin{proof} We define $\nu_n := |\pi_n(\alpha_0)-\alpha_0|_H
  %\downarrow 0$ as $n \to + \infty$.
%  

Let us start by testing equation \eqref{eq:probapprn} by $u_{\ne n}$ and getting
\begin{align*}
  & \left( \partial_t (\kappa\ast( z_{\ne n}-z_{0{\ne n}})),
    u_{\ne n}  \right)  + ( A  u_{\ne n},  u_{\ne n}) \\
  &\quad + (
  (\pi_n\circ \beta_{\epsi H})(u_{\ne n}), u_{\ne n}) =
    ((\pi_n\circ G)(u_{\ne n}),u_{\ne n})\quad \text{a.e.~in} \ (0,T).
    \end{align*}
Since $u_{\ne n},\, z_{\ne n} \in H_n$ a.e.~in $(0,T)$,  we drop 
$\pi_n$  in the above equation, and replace $z_{\ne n} $ by
$v_{\ne n}$  and $z_{0\ne n} $ by $v_{0\ne}$. This gives
\begin{align}
  &( \partial_t (\kappa\ast( v_{\ne n}-v_{0\ne})),
    u_{\ne n}) + \|  u_{\ne n}\|^2_V + ( \beta_{\epsi H}(u_{\ne n}), u_{\ne n}) =
    ( G(u_{\ne n}),u_{\ne n})  \label{eq:conv}
    \end{align}
a.e.~in $(0,T)$.  As $v_{0 \ne} \in D(\psi_{\eta_{\ne}})$  (see Lemma \ref{L:regu_data}, (iii)),  $v_{\ne n} ,u_{\ne n} \in L^2(0,T;H)$,  $\kappa\ast( v_{\ne n}-v_{0\ne}) \in W^{1,2}(0,T;H)$, $(\kappa\ast( v_{\ne
  n}-v_{0\ne}))(0) = 0$,   $\psi_{\eta_{ \ne 
 }}( v_{\ne n})\in L^1(0,T)$  (indeed, one can prove it as in the proof of Lemma \ref{L:regu_data}, (iii)),  and $u_{\ne n}\in \eta_{\ne 
  H}(v_{\ne n})$ a.e.~in $(0,T)$  (see~\eqref{v_nuen}),  by applying the nonlocal chain-rule
inequality \eqref{eq:chain} we get
$$\ell \ast \big( \partial_t(\kappa \ast (v_{\ne n} - v_{0\ne  })),u_{\nu 
  	 \epsi  n} \big)  \geq \int_{\Omega} 
    \haz{\eta_{\ne}} (v_{\ne n})\, \de x-\int_\Omega
    \haz{\eta_{\ne}}(v_{0\ne}) \,\de x\quad \text{a.e.~in} \ (0,T).$$
    Hence, convolving \eqref{eq:conv} with $\ell$ we get
\begin{align*}
  & \int_{\Omega} 
    \haz{\eta_{\ne}} (v_{\ne n})\, \de x-\int_\Omega
    \haz{\eta_{\ne}}(v_{0\ne }) \,\de x
    + \ell\ast \|u_{\ne n}\|^2_{V} +
    \ell\ast(\beta_{\epsi H}(u_{\ne n}), u_{\ne n}) \\
  &\quad \leq
    \ell\ast (G(u_{\ne n}), u_{\ne n})\quad \text{a.e.~in} \ (0,T).
\end{align*}
The above right-hand side can be controlled via \eqref{ass_lip2} as
\begin{align*}
   \ell\ast (G(u_{\ne n}), u_{\ne n})
  \stackrel{\eqref{ass_lip2}}{\leq} \ell \ast  \frac{1}{2c_V^2}   \| u_{\ne
  n}\|^2 + \ell\ast  C_G   \leq \ell\ast   \frac12  \| u_{\ne
  n}\|^2_V + C \quad \text{a.e.~in} \ (0,T).
\end{align*}
Accordingly, we get
\begin{align}
\MoveEqLeft{
\int_{\Omega} \haz{\eta_{\ne}} (v_{\ne n})\, \de x +  \ell \ast \frac12
  \|u_{\ne n}\|^2_{V}  + \ell \ast (\beta_{\epsi H}(u_{\ne n}),
                u_{\ne n}) 
}\nonumber\\
& \le \int_{\Omega}
  \haz{\eta_{\ne}}(v_{0\ne }) \,\de x +C\label{eq:bound0}
\end{align}
a.e.~in $(0,T)$.
%We now show that the first term in the right-hand side of
%\eqref{eq:bound0} is bounded. %We have 
% \[ \int_{\Omega} \haz{\eta_{\ne}}(\pi_n(v_{0\ne}))\, \de x = \int_\Omega \left ( \haz{\eta_{\ne}}(\pi_n(v_{0\ne}))\,-\haz{\eta_{\ne}}(v_{0\ne}) \right )\, \de x + \int_\Omega \haz{\eta_{\ne}}(v_{0\ne}) \,\de x=:I_1+I_2. \]
% We treat separately the two integrals $I_1$ and $I_2$. For $I_1$ we have
% \begin{align*}
%     I_1&\le \int_\Omega  \eta_{\ne}(\pi_n(v_{0\ne})) ( \pi_n(v_{0\ne})- v_{0\ne}) \,\de x \\
% & = \int_\Omega ( \pi_n(v_{0\ne})- v_{0\ne}) \big(\eta_{\ne}(\pi_n(v_{0\ne}))-\eta_{\ne}(v_{0\ne})+\eta_{\ne}(v_{0\ne})\big)\, \de x \\ 
%     & \le \|\pi_n(v_{0\ne})- v_{0\ne}\| \left( \frac{1}{\ne} \|\pi_n(v_{0\ne}) - v_{0\ne}\| + \|u_0\| \right) \\
%     & \le (2\ne \|u_0\| + \|\pi_n(\alpha_0)-\alpha_0\|) \left(
%       \frac{1}{\ne} (2\ne \|u_0\| +
%       \|\pi_n(\alpha_0)-\alpha_0\|)  + |u_0|_H \right)\\[1mm] 
%     & = \ne (2\|u_0\| +1 ) ( 3 \|u_0\|+1)\leq C\ne.
% \end{align*}
  Recalling  (iii) of Lemma \ref{L:regu_data}, we also find that
    \begin{align*}
%\sup_{\epsi,\nu \in (0,1)}
\int_{\Omega} \haz{\eta_{\ne}}(v_{0\ne}) \,\de x\leq C. 
    \end{align*}
Since all the terms in the left-hand side of \eqref{eq:bound0} are
nonnegative, each of them is  uniformly bounded. In
particular, we have that $\ell\ast \|u_{\ne n}\|_V^2\leq C$ a.e.~in
$(0,T)$. By convolving this with $\kappa$
we get that
\begin{equation}
  \label{eq:bound1}
  \| u_{\ne n}\|_{L^2(0,T;V)}\leq C.
\end{equation}
As $A  : V \to V^*$ is bounded we also get
\begin{equation}
  \label{eq:bound2}
  \| Au_{\ne n}\|_{L^2(0,T;V^*)}\leq C.
\end{equation}
Bound \eqref{eq:bound1} and the $\epsi^{-1}$-Lipschitz continuity of
 $\alpha_{\epsi }$
and $\beta_{\epsi }$  entail that 
\begin{equation}
  \label{eq:bound3n}
  \| \alpha_{\epsi H} (u_{\ne n})\|_{L^2(0,T;V)}+ \| \beta_{\epsi H}(u_{\ne n}) \|_{L^2(0,T;V)}\leq C(\epsi),
\end{equation}
which in particular also gives 
\begin{equation}
  \label{eq:bound4n}
  \| (\pi_n \circ \beta_{\epsi H})(u_{\ne n}) \|_{L^2(0,T;H)}\leq C(\epsi).
\end{equation}
Moreover,  the $(\nu +\epsi^{ -1})$-Lipschitz continuity of $\alpha_{\nu \epsi H}$  and $\alpha_{\ne H}(0) = 0$,  gives that
\begin{equation}
  \label{eq:bound5n}
  \| v_{\ne n}\|_{L^2(0,T;V)}+\| z_{\ne n}\|_{L^2(0,T; H)} \leq C(\epsi)
\end{equation}
 uniformly for  $n \in \N$ and  $\nu  \in (0,1)$  (see \eqref{v_nuen}). 
Bound \eqref{eq:Gbound} implies that
\begin{align}
  &\|(\pi_n\circ G)(u_{\ne n})\|_{L^2(0,T;H)} \leq
    \|G(u_{\ne n})\|_{L^2(0,T;H)} \nonumber \\
  &\quad \leq    \Lambda_g |\Omega|^{1/2 }  T^{1/2}+
 \Lambda_g \|u_{\ne n}\|_{L^2(0,T;H)}  \leq C . \label{eq:bound6n}
    \end{align}
 Finally, a comparison in \eqref{eq:probapprn} and
 bounds \eqref{eq:bound2}, \eqref{eq:bound4n}, and \eqref{eq:bound6n} entail
\begin{equation}
   \label{eq:bound5nn}
   \| \partial_t (\kappa\ast (z_{\ne n} - z_{0\ne n}))\|_{L^2(0,T;V^*)} \leq C(\epsi).
 \end{equation}

 \subsection{Passage to the limit as $n\to \infty$}\label{sec:limit}

The a-priori estimates from Section \ref{sec:estimates} allow us to pass to the limit as $n\to \infty$,  for $\epsi >0$ and $\nu\in (0,1)$ fixed. We start by controlling the time increments of $z_{\ne n} $. Using the short-hand notation
    \[ g_{\ne n}(t) :=  \partial_t( \kappa\ast(
        z_{\ne n}-z_{0\ne n}) )(t), \quad t \in (0,T),\]
         for  every  $h \in (0,T)$, we can estimate
   \begin{align}
       &\int_0^{T-h} \|z_{\ne n}(t+h)- z_{\ne n}(t)\|^2_{V^*}\, \de t
         \nonumber\\
     &\quad=  \int_0^{T-h} \|(\ell\ast g_{\ne n})(t+h)- (\ell\ast g_{\ne n})(t)\|^2_{V^*} \,\de t\nonumber\\
       &\quad= \left \| (\ell(\cdot+h)-\ell) \ast g_{\ne n} +
         \int_{\,\cdot\,}^{\,\cdot\,+h} \ell(\cdot + h-s) g_{\ne n}(s) \,\de s
         \right \|_{L^2(0,T-h; V^*)}^2\nonumber\\
     &\quad \leq 2 \left \| (\ell(\cdot+h)-\ell) \ast g_{\ne n}
       \right\|^2_{L^2(0,T-h; V^*)} \nonumber\\
     &\qquad + 2\left \| \int_{\,\cdot\,}^{\,\cdot\,+h} \ell(\cdot + h-s) g_{\ne n}(s) \,\de s
       \right \|_{L^2(0,T-h; V^*)}^2.\label{eq:incr}
   \end{align}
The first term in the above right-hand side can be bounded via the
Young convolution inequality \eqref{eq:young}  as follows:
\begin{align*}
  &2 \|( \ell(\cdot +t) - \ell)\ast g_{\ne n}\|^2_{L^2(0,T-h;V^*)}\stackrel{\eqref{eq:young}}{\leq}2\|\ell(\cdot+h)-\ell\|^2_{L^1(0,T)} \|g_{\ne
    n}\|^2_{L^2(0,T;V^*)}\\
  &\quad\leq C (\epsi)  \|\ell(\cdot+h)-\ell\|^2_{L^1(0,T)} , 
  \end{align*}
   where we used \eqref{eq:bound5nn} in the last inequality. 
By defining
$$\ell_h (\sigma)
:=
\left\{
\begin{array}{ll}
  \ell(\sigma) \quad &\text{for} \ \sigma \in (0,h),\\[1mm] 
  0 \quad &\text{for} \ \sigma \in [h,T),\\
\end{array}
\right.
$$
we can control the second term in the  right-hand  side of
\eqref{eq:incr} as follows:
\begin{align*}
  &2\left \| \int_{\,\cdot\,}^{\,\cdot\,+h} \ell(\cdot + h -s) g_{\ne n}(s) \,\de s
  \right \|_{L^2(0,T-h; V^*)}^2\\
  &\quad \leq
  2\int_0^{T-h}\left(\int_{t}^{t+h} \ell(t+h-s) \|g_{\ne
    n}(s)\|_{V^*} \,\de s\right)^2\de t\\
  &\quad =  2\int_0^{T-h}\left(\int_{0}^{t+h} \ell_h(t +h -s) \|g_{\ne
    n}(s)\|_{V^*} \,\de s\right)^2\de t\\
  &\quad  = 2 \big\| \ell_h * \|g_{\ne n}\|_{V^*} \big\|_{L^2(h,T)}^2 \leq 2 \| \ell_h\|_{L^1(0, T )}^2  \|g_{ \ne n} \|^2_{L^2(0, T ;V^*)} \stackrel{ \eqref{eq:bound5nn}}{\leq} C(\epsi) 
    \| \ell\|^2_{L^1(0,h)}.
\end{align*}
All in all, we have proved that
$$\int_0^{T-h} \|z_{\ne n}(t+h)- z_{\ne n}(t)\|^2_{V^*}\, \de t \leq
 C (\epsi) \left( \|\ell(\cdot+h)-\ell\|^2_{L^1(0,T)} +
\|\ell\|_{L^1(0,h)}^2 \right).$$
As $\ell \in L^1(0,T)$ this entails that
  \begin{equation}
    \limsup_{h \to 0} \sup_{n \in \N}\int_0^{T-h} \|z_{\ne n}(t+h)-
  z_{\ne n}(t)\|^2_{V^*}\, \de t =0. \label{eq:wvar}
  \end{equation} 

As a consequence of bounds \eqref{eq:bound1}--\eqref{eq:bound5n}
and of \eqref{eq:wvar}, an application of the Aubin--Lions lemma
entails that,  up to a subsequence, as $n \to \infty$, 
\begin{align}
  &u_{\ne n} \weakto u_{\ne}\quad \text{in} \ L^2(0,T; V),\label{eq:uweak}\\
   &v_{\ne n} \weakto v_{\ne} \quad \text{in} \  L^2(0,T;  V ), \label{eq:vweak}\\
    &z_{\ne n} \to \bar v_{\ne} \quad \text{in} \  L^2(0,T;
      V^*), \label{eq:wstrong}\\
   &z_{\ne n} \weakto \bar v_{\ne} \quad \text{in} \  L^2(0,T; H), \label{eq:wweak}\\   
%    &(\pi_n\circ \beta_{\epsi H} )(u_{\ne n}) \weakto w_{\ne}
%      \quad\text{in} \  L^2(0,T; H),\label{eq:gammaweak}\\
   &\partial_t(\kappa\ast (z_{\ne n} - z_{0\ne n})) \weakto
  \xi_{\ne} \quad \text{in} \ \ L^2(0,T;V^*).  \label{eq:xiweak}
\end{align}

It is a standard matter to check that $\bar v_{\ne} = v_\ne$. Indeed, for all $y\in L^2(0,T,H)$ one has
\begin{align*}
  &\int_0^T (\bar v_\ne,y)\, \de t \stackrel{\eqref{eq:wweak}}{=}
    \lim_{n\to \infty} \int_0^T (z_{\ne n},y)\, \de t =
    \lim_{n\to \infty} \int_0^T (v_{\ne n},\pi_n(y))\, \de t \\
  &\quad = \lim_{n\to \infty} \left(\int_0^T
    (v_{\ne n},y)\, \de t +\int_0^T
    (v_{\ne n},\pi_n(y) -y)\, \de  t\right) \stackrel{\eqref{eq:vweak}}{=} \int_0^T (v_\ne,y)\, \de t
\end{align*}
where we used bound \eqref{eq:bound5n}  and the fact that $\pi_n(y)\to y$ in $L^2(0,T;H)$  as $n \to \infty$.  Moreover,  from the linearity and maximal monotonicity of the map $\mathcal B$ defined with $X =  V^*$ and $p = 2$ (see \S \ref{sec:main14}), from the fact that
$z_{0\ne n}=\pi_n(v_{0\ne}) \to v_{0\ne}$ in $H$  as $n \to \infty$,  we readily obtain  $\kappa * (v_{\nu \epsi} - v_{0\ne}) \in W^{1,2}(0,T;V^*)$ with $(\kappa * (v_{\nu \epsi} - v_{0\ne}))(0) = 0$ and  $\xi_\ne = \partial_t(\kappa\ast (v_\ne- v_{0\ne}))$.  Furthermore,  we note that 
\begin{align}
  Au_{\ne n}\weakto Au_\ne\quad \text{in} \ \ L^2(0,T;V^*). \label{eq:Aweak}
\end{align}

Owing to convergences \eqref{eq:uweak} and \eqref{eq:wstrong} we have that
\begin{align}
  &\int_0^T (v_{\ne n},u_{\ne n}) \, \de t=  \int_0^T (z_{\ne n},u_{\ne n}) \,
  \de t =  \int_0^T \la z_{\ne n},u_{\ne n}\ra\, \de t \nonumber \\
  &\quad \to \int_0^T \la
  v_\ne, u_\ne\ra \, \de t =\int_0^T (
  v_\ne, u_\ne) \, \de t .\label{eq:toli}
\end{align}
By the classical tool~\cite[Proposition~2.5, p.~27]{Brezis73},  it follows  from \eqref{v_nuen}  that $v_{\ne}\in \partial
\Psi_{\alpha_{\ne}}(u_{\ne})$ or, equivalently,
$v_\ne = \alpha_\ne(u_\ne)$ a.e.~in $Q$.

% Together with the $\beta$-convergence
% $\psi_{\alpha_\ne} \to \psi_\alpha$, and the pointwise bound
% $  \limsup_\nu\psi_{\alpha_\ne}(y) \leq \psi_\alpha(y)$, for all $y
% \in H$,
The lower semicontinuity  and convexity of $\Psi_{\alpha_\ne}$, convergences \eqref{eq:uweak}--\eqref{eq:vweak}, and \eqref{eq:toli} ensure that
\begin{align*}
  &\Psi_{\alpha_\ne}(u_\ne)\leq  \liminf_{n\to \infty}
    \Psi_{\alpha_\ne}(u_{\ne n}) \leq \limsup_{n\to \infty} \Psi_{\alpha_\ne}(u_{\ne n}) \\
  &\quad\leq  \limsup_{n\to \infty}\left( \Psi_{\alpha_\ne}(u_\ne)+ \int_0^T(v_{\ne n},  u_{\ne n}-u_\ne)\, \de t
    \right) = \Psi_{\alpha_\ne}(u_\ne).
\end{align*}
This in particular proves that
$$\Psi_{\alpha_\ne}(u_{\ne n})\to
\Psi_{\alpha_\ne}(u_\ne).$$ 
Since the functional $\Psi_{\alpha_\ne}$ is
strictly convex, we deduce from convergence \eqref{eq:uweak} and
\cite[Theorem~3~(ii)]{Visintin} that indeed
\begin{equation}
  \label{eq:ustrong}
  u_{\ne n} \to u_\ne \quad \text{in} \ L^2(0,T;H).
\end{equation} Since $\beta_{\epsi H}$ is $ \epsi^{-1}$-Lipschitz continuous in $H$, we also find that
\begin{equation}\label{eq:noproj}
\beta_{\epsi H}(u_{\ne n}) \to \beta_{\epsi H}(u_{\ne}) \quad \mbox{ in } L^2(0,T;H).
\end{equation}
Here we note that $\beta_{\epsi H}(u_{\ne}) \in L^2(0,T;V)$ as $\beta_{\epsi}$ is $\epsi^{-1}$-Lipschitz continuous in $\R$ and $u_\ne\in L^2(0,T;V)$.  
The strong convergence \eqref{eq:ustrong} and the Lipschitz continuity
in \eqref{eq:ass_lip3} allow us to check that 
\begin{align*}
  & \| (\pi_n \circ G)(u_{\ne n}) - G(u_\ne)\|_{L^2(0,T;H)}\nonumber\\
  &\quad \leq \| (\pi_n
  \circ G)(u_{\ne n}) - (\pi_n \circ G) (u_\ne))\|_{L^2(0,T;H)}+\| (\pi_n
   \circ G)(u_\ne) - G(u_\ne)\|_{L^2(0,T;H)}\\
  &\quad \leq \Lambda_g\| u_{\ne n} - u_\ne \|_{L^2(0,T;H)}+\| (\pi_n
  \circ G)(u_\ne) - G(u_\ne)\|_{L^2(0,T;H)}\to 0,
\end{align*}
so that
\begin{equation}
  \label{eq:Gstrong} (\pi_n \circ G)(u_{\ne n}) \to G(u_\ne)  \quad
  \text{in} \ L^2(0,T;H).
\end{equation}
 Moreover, one can similarly verify that
\begin{equation}\label{eq:gammaweak}
(\pi_n \circ \beta_{\epsi H})(u_{\ne n}) \to \beta_{\epsi H}(u_{\ne}) \quad \mbox{  in  } L^2(0,T;H).
\end{equation}

\begin{comment}
Moreover, since ${\rm id}_H -\pi_n$, is a self-adjoint operator, 
for any $y \in L^2(0,T; H)$ we have  that
\begin{align*}
&\int_0^T (\beta_{\epsi H}(u_{\ne n}), y) \, \de t = \int_0^T ({\rm id}_H -\pi_n)(\beta_{\epsi H}(u_{\ne n})),y) \, \de t + \int_0^T ((\pi_n \circ \beta_{\epsi H})(u_{\ne n}),y) \, \de t \\
&\quad = \int_0^T  ( \beta_{\epsi H}(u_{\ne n}), ({\rm id}_H -\pi_n)y ) \, \de t + \int_0^T ((\pi_n \circ \beta_{\epsi H})(u_{\ne n}),y) \, \de t  \to \int_0^T ( w_{\ne}, y) \, \de t,
\end{align*}
where we also used  the bound  \eqref{eq:bound3n}.
We deduce that 
\begin{equation}\label{eq:noproj}
\beta_{\epsi H}(u_{\ne n}) \weakto w_{\ne} \quad \text{in}\ L^2(0,T;H).
\end{equation}

 Furthermore,  the strong convergence \eqref{eq:ustrong}  along with \eqref{eq:uweak}  also ensures that 
\begin{align*}
  &\int_0^T (\beta_{\epsi H}(u_{\ne n}),u_{\ne n})\, \de t \to
    \int_0^T(w_\ne,u_\ne) \, \de t
\end{align*}
which,  together with the convergences in \eqref{eq:noproj}, \eqref{eq:ustrong}, and
\cite[Proposition 2.5, p.~27]{Brezis73}, yields $w_\ne \in \partial \Psi_{\beta_\epsi}(u_{\ne})$
or, equivalently, $w_\ne = \beta_\epsi(u_\ne)$ a.e.~in $Q$. Note
that  $w_\ne\in L^2(0,T;V)$ as $\beta_\epsi$ is $\epsi^{-1}$-Lipschitz
continuous and $u_\ne\in L^2(0,T;V)$. 
\end{comment}

Owing to convergences \eqref{eq:xiweak}, \eqref{eq:Aweak}, \eqref{eq:Gstrong}, and \eqref{eq:gammaweak}  we can
pass to the limit in equation \eqref{eq:probapprn} and in bound \eqref{eq:bound1}. Taking
also into account the uniqueness argument of Section \ref{sec:uniqueness}, we have proved the following proposition. 

\begin{proposition}[Well-posedness of the regularized problem]\label{cor:reg} 
Under assumptions \eqref{Azero}--\eqref{eq:init}, for all $\epsi>0$ and $\nu \in(0,1)$,  set $v_{0\ne}=\nu u_{0\epsi} + v_{0\epsi}$ with $u_{0\epsi}$ and $v_{0\epsi}$ defined in \eqref{eq:u0}--\eqref{eq:v0}. Then  there exists a unique triplet
$(u_\ne,v_\ne,w_\ne)\in  L^2(0,T;V\times V  \times V )$ 
  such that  $\kappa * (v_{\nu \epsi} - v_{0\ne}) \in W^{1,2}(0,T;V^*)$ with $(\kappa * (v_{\nu \epsi} - v_{0\ne}))(0) = 0$  and 
  \begin{align}
    &\partial_t (\kappa \ast (v_\ne - v_{0\ne}) )
       + A u_\ne + w_\ne =G(u_\ne) \quad \text{in} 
      \ V^*, \, \text{a.e.~in} \ (0,T), 
    \label{eq:prob5e}\\
    & v_\ne = \alpha_\ne(u_\ne) \quad \text{a.e.~in} \ Q, 
    \label{eq:prob2e}\\
    &w_\ne = \beta_\epsi(u_\ne)\quad \text{a.e.~in}   \  Q.
    \label{eq:prob4e}
  \end{align}
Moreover, there exists a constant $ C_1>0$ independent of $\nu  \in (0,1)$ and $\epsi  \in (0,1)$
such that
\begin{align}
  &\| u_{\ne }\|_{L^2(0,T;V)}
  \leq C_1.\label{eq:bounds}
\end{align}
\end{proposition}

%The uniqueness of $(u_\ne,v_\ne,w_\ne)$ solving \eqref{eq:prob5e}--\eqref{eq:prob4e} entails in particular that no extraction of subsequences is actually needed for the convergences \eqref{eq:uweak}- \eqref{eq:gammaweak}  above to hold. [if $\kappa$ is convex]

\subsection{Additional a-priori estimates  for  the regularized problem}\label{sec:estimates2}

With the aim of passing to the limit in the regularizations, we
complement bound \eqref{eq:bounds} by proving 
some additional estimates  for  the solution
$(u_{\ne},v_{\ne},w_{\ne})$ of the regularized problem
\eqref{eq:prob5e}--\eqref{eq:prob4e}.
 
 Recalling that $w_\ne=\beta_{\epsi H}(u_{\ne })\in L^2(0,T;V)$
and that $|\beta_\epsi( r )|\leq 1/\epsi$ for all $ r  \in \R$, we readily
check that $y_\ne:=|w_\ne|^{ q-2}w_\ne \in L^2(0,T;V)$  with $q > 2$ given in~(\ref{ass_lip}),  as well. We can hence
test \eqref{eq:prob5e} by $y_\ne$ and obtain
\begin{align}
  &\langle \partial_t(\kappa\ast(v_{\ne } - v_{0\ne})),
 y_\ne\rangle+ \la A u_{\ne },
    y_\ne\ra  + \|w_\ne\|_{L^{ q}(\Omega)}^{ q} = ( G(u_{\ne }),
    y_\ne)  \quad \text{a.e.~in} \ (0,T).\label{eq:test1} 
    \end{align}
  The second term in the left-hand side of \eqref{eq:test1} is nonnegative, since  we see that
  \begin{align}
  &  \la A u_{\ne },
    y_\ne\ra = \int_\Omega \nabla u_{\ne }\cdot
    \nabla (|\beta_\epsi(u_{\ne })|^{ q -2}\beta_\epsi(u_{\ne })) \, \de
    x \nonumber\\
    &\quad = ( q -1)\int_\Omega|\beta_\epsi(u_\ne)|^{ q -2}
    \beta_\epsi'(u_{\ne })|\nabla u_{\ne }|^2\, \de x\geq
    0.\label{eq:Apos}
    \end{align}
  The right-hand side of \eqref{eq:test1} can be
    bounded  from  the sublinearity assumption in \eqref{ass_lip} as
    \begin{align}
      &(G(u_{\ne }),
     y_\ne)\leq \frac{1}{ q'} \|w_\ne\|^{q}_{L^{ q}(\Omega)} +
        \frac{1}{ q}  \|G(u_{\ne })\|^{q}_{L^{ q}(\Omega)} \nonumber\\
      &\quad \leq \frac{1}{q'}  \|w_\ne\|^{q}_{L^{ q}(\Omega)} + C ( 1
        +   \| u_\ne \|^2) \quad
        \text{a.e.~in} \ (0,T)
        \label{eq:estG}.
        \end{align}
  Using  the fact  that $\alpha_{\ne H}$ is
invertible and recalling that $\beta_{\epsi  q}( r ) :=
|\beta_\epsi( r )|^{ q - 2  }\beta_\epsi( r )$ for all $ r  \in \R$  as in Lemma \ref{L:regu_data}(iv), one can write 
  $$y_\ne=\beta_{\epsi  q  H} (u_{\ne })=(\beta_{\epsi  q  H}\circ \alpha_{\ne
    H}^{-1}) (\alpha_{\ne H}(u_{\ne }))= (\beta_{\epsi  q  H}\circ
  \alpha_{\ne H}^{-1})(v_{\ne }).$$
  As $v_{0\ne} \in D(\psi_{\beta_{\epsi  q } \circ \alpha_{\ne}^{-1}})$
  (see Lemma \ref{L:regu_data} (iv)),  $v_{\ne},\, y_\ne \in L^2(0,T;V)$,  $\kappa\ast (v_{\ne } - v_{0\ne}) \in W^{1,2}(0,T;V^*)$ with $(\kappa\ast (v_{\ne } - v_{0\ne}))(0)=0$,  $\psi_{\beta_{\epsi  q } \circ \alpha_{\ne}^{-1}}(v_{\ne }) \in
  L^1(0,T)$  (indeed, one can prove it as in the proof of Lemma
  \ref{L:regu_data} (iv))),  and $y_\ne=\beta_{\epsi  q  H} (u_{\ne })=
  (\beta_{\epsi  q  H}
  \circ \alpha_{\ne H}^{-1})(v_{\ne }) $ a.e.~in $(0,T)$, by
  applying the nonlocal chain-rule
inequality \eqref{eq:chain2} we get
    \begin{align*}
      &\ell\ast \big\langle \partial_t(\kappa\ast(v_{\ne } - v_{0\ne})),
     y_\ne \big\rangle \geq   
      \int_\Omega \haz{\beta_{\epsi  q }\circ
      \alpha_{\ne}^{-1}}(v_{\ne })\, \de x - \int_\Omega
        \haz{\beta_{\epsi  q }\circ
      \alpha_{\ne}^{-1}}(v_{0\ne})\, \de x 
    \end{align*}
    a.e.~in $(0,T)$. Hence, by convolving \eqref{eq:test1} with
    $\ell$, also using \eqref{eq:estG}  and Lemma \ref{L:regu_data}(iv),  we deduce  that 
 \begin{align*}   &\int_\Omega \haz{\beta_{\epsi  q}\circ
      \alpha_{\ne}^{-1}}(v_{\ne })\, \de x
    +\ell\ast \frac{1}{ q}\|w_\ne \|_{L^{ q}(\Omega)}^{ q} \leq  C + \ell\ast C(1+\|u_{\ne}\|^2).
 \end{align*}
 As $\haz{\beta_{\epsi  q}\circ
      \alpha_{\ne}^{-1}}\geq 0$, one has that  
 $$\ell \ast   \frac{1}{ q}\|w_\ne \|_{L^{ q}(\Omega)}^{ q} \leq  C + \ell\ast C(1+\|
                    u_{\ne}\|^2).$$
By convolving with $\kappa$ and using bound \eqref{eq:bounds} we deduce that
 \begin{equation}
   \label{eq:bound3}
\|w_\ne\|_{L^{ q}(Q)} \leq C.
\end{equation}

 Testing \eqref{eq:prob5e} by $v_{\ne } \in L^2(0,T;V)$, we have
\begin{align}
  &\langle \partial_t(\kappa \ast(v_{\ne } - v_{0\ne })),
v_{\ne }\rangle + \la A u_{\ne },
    v_{\ne }\ra +
(  w_\ne, v_{\ne }) = (G(u_{\ne }),
   v_{\ne }) \quad \text{a.e.~in} \ (0,T).\label{eq:test2}
\end{align}
The second and the third  terms  in the left-hand side of \eqref{eq:test2} are nonnegative since $\la A u_{\ne },
    v_{\ne }\ra  = \la A u_{\ne },\alpha_\ne(u_{\ne })\ra \geq 0$,
    see \eqref{eq:Apos}, and
    $$(  w_\ne, v_{\ne }) = (
    \beta_{\epsi H}(u_{\ne }), \alpha_{\ne H}(u_{\ne })) =
    \int_\Omega \beta_\epsi(u_{\ne })\,\alpha_\ne(u_{\ne })\, \de
    x\geq 0$$
    as $\beta_\epsi( r)\,\alpha_\ne( r)\geq 0$ for all $ r\in \R$. Moreover, the right-hand side of \eqref{eq:test2} can be
    bounded as in \eqref{eq:estG}  with the aid of   \eqref{eq:Gbound}, namely, 
    \begin{align*}
     (G(u_{\ne }), v_{\ne }) &\leq \frac 1 {4 \|\ell\|_{L^1(0,T)}}
     \| v_{\ne } \|^2 + \|\ell\|_{L^1(0,T)} \|G(u_{\ne })\|^2\\
     &  \leq \frac 1 {4 \|\ell\|_{L^1(0,T)}} \| v_{\ne }\|^2 + 2 \|\ell\|_{L^1(0,T)} \left(  \Lambda_g^2|\Omega|  + \|u_{\ne }\|^2 \right).
    \end{align*} 
    As  $v_{\ne }\in L^2(0,T;V)$ and  $\kappa \ast (v_{\ne }-v_{0\ne})
    \in W^{1,2}(0,T;V^*)$ along with $(\kappa \ast (v_{\ne
    }-v_{0\ne}))(0)=0$,  an application   of   the nonlocal chain-rule \eqref{eq:chain2} for the functional
    $v\mapsto \psi_{\rm id}(v)= \|v\|^2/2$ ensures that 
  \begin{align*}
    & \ell \ast \langle\partial_t(\kappa \ast(v_{\ne } - v_{0\ne})),
      v_{\ne }\rangle\geq \frac12\|v_{\ne }\|^2 -
      \frac12\|v_{0\ne}\|^2\quad \text{a.e.~in} \ (0,T).
  \end{align*}
  Taking the convolution of \eqref{eq:test2} with $\ell$ and using the above
  bounds we hence have that  
  $$\frac12\|v_{\ne }\|^2 -
      \frac12\|v_{0\ne}\|^2 \leq \ell \ast  \frac1 {4\|\ell\|_{L^1(0,T)}}\|
  v_{\ne }\|^2 +\ell \ast 2 \|\ell\|_{L^1(0,T)} \left(  \Lambda_g^2|\Omega| + \| u_{\ne}\|^2\right).$$
  Recalling that
 $\|v_{0\ne}\|\leq \nu\|u_{0\epsi}\| + \| v_{0\epsi}\|\leq C$ by \eqref{eq:v0conv}--\eqref{eq:u0conv} and using
 the Young convolution inequality \eqref{eq:young}  along with bound \eqref{eq:bounds} we conclude that
 % \begin{align*}
 %  \frac12 \|v_{\ne }\|^2 \leq C + C  \ell * \| v_{\ne}\|^2  + C  \ell * \| u_{\ne}\|^2
 %\end{align*}
 %Using and  Young's convolution  lemma we conclude that 
 \begin{equation}
   \label{eq:bound4}
   \| v_{\ne }\|_{L^{ 2}(0,T;H)} \leq C.
 \end{equation}
 
 Finally, a comparison in \eqref{eq:prob5e}, bounds \eqref{eq:Gbound},
 \eqref{eq:bounds},  \eqref{eq:bound3}, and the boundedness of $A  : V \to V^*$   ensure that
\begin{equation}
   \label{eq:bound5}
   \| \partial_t (\kappa \ast (v_{\ne } - v_{0\ne }))\|_{L^2(0,T;V^*)} \leq C.
 \end{equation}

\subsection{Passage to the limit as $\epsi \to 0$}\label{sec:limit2}

Bound \eqref{eq:bounds} and the a-priori estimates from
 Section \ref{sec:estimates2} allow us to pass
 to the limit as  $\epsi \to 0$ and $\nu\to0$. In fact,
 these limits can be taken simultaneously. Still, we prefer to pass to
 the limit as $\epsi\to 0$ first, in order to obtain an intermediate existence result,
 which could be of independent interest. The parameter
 $\nu \in (0,1)$ is hence kept fixed in this section. The limit  as  $\nu \to
 0$ is discussed in Section \ref{sec:limit3} below.  

 From  bounds 
\eqref{eq:bounds}, \eqref{eq:bound3}, and
\eqref{eq:bound4}--\eqref{eq:bound5}, by passing to  the limit (up to a subsequence) as  $\epsi \to 0$, we get
\begin{align}
  &u_{\ne } \weakto u_{\nu}\quad \text{in} \ L^2(0,T; V),\label{eq:uweake}\\
   &v_{\ne }  \weakto  v_{\nu} \quad \text{in} \  L^{ 2}(0,T; H), \label{eq:vweake}\\ 
    &w_\ne \weakto w_{\nu}    \quad\text{in} \   L^{ q}(\Omega \times
      (0,T)),  \label{eq:gammaweake}\\
   &\partial_t(\kappa\ast (v_{\ne } - v_{0\ne })) \weakto
  \xi_{\nu} \quad \text{in} \ \ L^2(0,T;V^*) . \label{eq:xiweake}
\end{align}
Arguing as in Section \ref{sec:limit}, bound \eqref{eq:bound5}
guarantees that 
the time increments of
$v_{\ne}$ can be controlled as
 \begin{equation}
    \limsup_{h \to 0} \sup_{\epsi \in (0,1)}\int_0^{T-h} \|v_{\ne}(t+h)-
  v_{\ne }(t)\|^2_{V^*}\, \de t =0.\label{eq:vvar}
\end{equation}
The Aubin--Lions lemma then gives 
\begin{equation}
  \label{eq:vstronge}
v_{\ne } \to v_{\nu} \quad \text{in} \  L^{ 2 }(0,T; V^*).
\end{equation}
%by possibly further extracting without relabeling. 
Together with convergence \eqref{eq:uweake},  this ensures that
\begin{equation}
  \int_0^T (v_\ne,u_\ne)\, \de t = \int_0^T \la v_\ne,u_\ne\ra\, \de
t\to  \int_0^T \la v_\nu,u_\nu\ra\, \de
t = \int_0^T ( v_\nu,u_\nu)\, \de
t.\label{eq:tolie}
\end{equation}
Recall that we have the convergence \eqref{eq:attouch2}, in
particular, $\partial \Psi_{{ \alpha}_\ne} \to \partial \Psi_{{ \alpha}_\nu}$ in the graph
sense in $L^2(0,T;H)$. Given the
convergences \eqref{eq:uweake}--\eqref{eq:vweake} and
\eqref{eq:tolie}, the extension \eqref{eq:attouch} of~\cite[Proposition~3.59, p.~361]{Attouch} guarantees that $v_\nu \in \partial \Psi_{{ \alpha}_\nu}(u_\nu)$,
namely, $v_\nu \in \alpha_\nu(u_\nu)$ a.e.~in $Q$.

At the same time, 
$\Psi_{\alpha_\ne} \to \Psi_{\alpha_\nu}$ in the Mosco sense in $L^2(0,T;H)$, so that
convergence \eqref{eq:uweake} and the $\liminf$ inequality
\eqref{eq:liminf} imply that
\begin{equation}\label{eq:liminfee}
 \Psi_{\alpha_\nu}(u_\nu) \leq
    \liminf_{\epsi \to 0} \Psi_{\alpha_\ne}(u_{\ne }). 
\end{equation}
    Hence, we can use the pointwise upper bound $\Psi_{\alpha_\ne}(y) \leq
\Psi_{\alpha_\nu}(y)$ for all $y  \in  L^2(0,T;H)$ and all $\epsi>0$  (indeed, we have $\widehat{\alpha_{\ne}} \leq \widehat{\alpha}_\nu$ in $\R$), and limit \eqref{eq:tolie} in order to get
\begin{align*}
  &\Psi_{\alpha_\nu}(u_\nu) \stackrel{ \eqref{eq:liminfee}}{\leq}
    \liminf_{\epsi \to 0} \Psi_{\alpha_\ne}(u_{\ne }) \leq \limsup_{\epsi \to 0} \Psi_{\alpha_\ne}(u_{\ne }) \\
  &\quad\leq  \limsup_{\epsi \to 0}\left( \Psi_{\alpha_\ne}(u_\nu) + \int_0^T(v_{\ne },  u_{\ne }-u_\nu)\, \de t
    \right)  \le  \Psi_{\alpha_\nu}(u_\nu),
\end{align*}
which gives the convergence $\Psi_{\alpha_\ne}(u_{\ne }) \to \Psi_{\alpha_\nu}(u_\nu)$  as $\epsi \to 0$.  We can similarly verify that $\Psi_{\alpha_\epsi}(u_{\ne}) \to \Psi_{\alpha}(u_\nu)$  as $\epsi \to 0$. Indeed, noting that $\alpha_\epsi(u_{\nu \epsi}) = v_{\nu \epsi} - \nu u_{\nu \epsi} \weakto v_\nu - \nu u_\nu =: a_\nu$ weakly in $L^2(0,T;H)$, we have
\begin{align*}
\lefteqn{
\limsup_{\epsi \to 0} (\alpha_\epsi(u_{\nu\epsi}),u_{\nu\epsi})
}\\
&\leq \limsup_{\epsi \to 0} (\alpha_{\nu\epsi}(u_{\nu\epsi}),u_{\nu\epsi}) - \liminf_{\epsi \to 0} \nu \|u_{\nu\epsi}\|^2\\
&\leq (v_\nu, u_\nu) - \nu \|u_\nu\|_H^2
= (a_\nu + \nu u_\nu, u_\nu) - \nu \|u_\nu\|_H^2 = (a_\nu,u_\nu),
\end{align*}
which also enables us to check $\Psi_{\alpha_\epsi}(u_{\ne}) \to \Psi_{\alpha}(u_\nu)$  as $\epsi \to 0$ in a similar fashion. Therefore it follows that 
    \begin{align*}
   &\limsup_{\epsi \to 0}   \frac{\nu}{2} \| u_{\ne}\|_{L^2(0,T;H)}^2 =\limsup_{\epsi \to 0}  \big(
      \Psi_{\alpha_\ne}(u_{\ne }) -
      \Psi_{\alpha_\epsi}(u_{\ne})\big)\\[1.5mm]
      &\quad \leq \limsup_{\epsi \to 0} \Psi_{\alpha_\ne}(u_{\ne })
        -\liminf_{\epsi \to 0} \Psi_{\alpha_\epsi}(u_{\ne}) \\
      &\quad \leq \Psi_{\alpha_\nu}(u_{\nu }) - \Psi_{\alpha}(u_{\nu})= \frac{\nu}{2}  \| u_{\nu}\|_{L^2(0,T;H)}^2,
    \end{align*}
which, along with convergence \eqref{eq:uweake}  and  the uniform convexity of $\|\cdot\|_{L^2(0,T;H)}$,  implies  that
\begin{equation}
  \label{eq:ustronge}
  u_{\ne} \to u_\nu \quad \text{in} \ L^2(0,T;H).
\end{equation} 

Moving from this strong convergence,  the continuity of $G$ in \eqref{ass_lip0} gives that $G(u_\ne) \to G(u_\nu)$ in $L^2(0,T;H)$
and  we also have 
$$\int_0^T (w_\ne,u_\ne)\, \de t \to \int_0^T (w_\nu,u_\nu)\, \de t.$$
As $\partial \Psi_{\beta_\epsi} \to \partial \Psi_{\beta}$ in the
graph sense in $L^2(0,T;H)$, see \eqref{eq:attouch2}, this last
convergence and convergences \eqref{eq:uweake} and \eqref{eq:gammaweake}
allow us to apply again the extension \eqref{eq:attouch} of~\cite[Proposition~3.59, p.~361]{Attouch} and obtain $w_\nu \in \partial \Psi_{\beta}(u_\nu)$, that is,
$w_\nu \in \beta(u_\nu)$ a.e.~in $Q$. Using  the linear maximal monotonicity and $m$-accretivity of $A$ and $\mathcal B$ with $X = V^*$ and $p = 2$ (see \S \ref{sec:main14}), respectively,  one can pass to the limit as $\epsi \to 0$ in equation
\eqref{eq:prob5e} as well as in bounds 
\eqref{eq:bounds}, \eqref{eq:bound3}, and
\eqref{eq:bound4}--\eqref{eq:bound5}, eventually
obtaining the following.

\begin{proposition}[Well-posedness of the regularized problem for
  $\nu\in (0,1)$]\label{cor:reg2} 
Under assumptions  \eqref{Azero}--\eqref{eq:init},  for all $\nu\in (0,1)$,  set $v_{0\nu}:=\nu u_0+v_0$. Then, there exists
  a triplet
$(u_\nu,v_\nu,w_\nu)\in  L^2(0,T;V\times H \times H)$ 
  such that  $\kappa \ast (v_\nu - v_{0 \nu}) \in W^{1.2}(0,T; V^*)$ along with $(\kappa \ast (v_\nu - v_{0 \nu}))(0)=0$ and  $w_\nu \in L^{ q}(Q)$,  and 
  \begin{align}
    &\partial_t (\kappa \ast (v_\nu - v_{0\nu}) )
       + A u_\nu + w_\nu =G(u_\nu) \quad \text{in} 
      \ V^*, \, \text{a.e.~in} \ (0,T), 
    \label{eq:prob5n}\\
    & v_\nu \in \alpha_\nu(u_\nu) \quad \text{a.e.~in} \ Q, 
    \label{eq:prob2n}\\
    &w_\nu \in \beta(u_\nu)\quad \text{a.e.~in}   \  Q.
    \label{eq:prob4n}
  \end{align}
Moreover, there exists a constant $C_2>0$ independent of $\nu$  
such that
\begin{align}
  &\| u_{\nu }\|_{L^2(0,T;V)}+\| v_{\nu }\|_{L^2(0,T;H)}+ \|
    w_{\nu }\|_{ L^{ q}(Q) }\nonumber\\
  &\quad + \| \partial_t (\kappa\ast (v_{\nu} - v_{0\nu}))\|_{L^2(0,T;V^*)} 
  \leq  C_2.\label{eq:boundsn}
\end{align}
\end{proposition}

\subsection{Passage to the limit as $\nu\to 0$}\label{sec:limit3}

We now pass to the limit as $\nu \to 0$  for the  solutions
$(u_\nu,v_\nu,w_\nu)$ to problem \eqref{eq:prob5n}--\eqref{eq:prob4n}  for  proving the existence
of a solution to \eqref{eq:prob5}--\eqref{eq:prob4}.  %To this end, we further suppose that $\beta$ is single-valued, and hence, it is continuous in $\mathbb{R}$.  
 Owing  to bounds
\eqref{eq:boundsn} we can extract  a (not relabeled) subsequence   
such that
\begin{align}
  &u_{\nu } \weakto u\quad \text{in} \ L^2(0,T; V),\label{eq:uweakee}\\
   &v_{\nu } \weakto v \quad \text{in} \  L^2(0,T; H), \label{eq:vweakee}\\ 
    &w_\nu \weakto w    \quad\text{in} \  L^{ q}(Q)
      , \label{eq:gammaweakee}\\
   &\partial_t(\kappa\ast (v_{\nu } - v_{0\nu })) \weakto
  \xi \quad \text{in} \ \ L^2(0,T;V^*) .  \label{eq:xiweakee}
\end{align}
Again,  thanks to the bound on  $ \partial_t (\kappa\ast (v_{\nu}
- v_{0\nu}))$ in ${L^2(0,T;V^*)}$  from \eqref{eq:boundsn}  we can control the time increments of $v_\nu$ as 
 \begin{equation*}
    \limsup_{h \to 0} \sup_{\nu \in (0,1)}\int_0^{T-h} \|v_{\nu}(t+h)-
  v_{\nu }(t)\|^2_{V^*}\, \de t =0,%\label{eq:vvarn}
\end{equation*}
so that, possibly extracting again without relabeling,  the Aubin--Lions lemma gives 
\begin{equation}
  \label{eq:vstrongn}
v_{\nu } \to v \quad \text{in} \ L^2(0,T; V^*).
\end{equation}
By reproducing the argument in \eqref{eq:tolie} we obtain
\begin{equation}
  \int_0^T (v_\nu,u_\nu)\, \de t \to \int_0^T ( v,u)\, \de
t.\label{eq:tolin}
\end{equation}
Using the convergence $\partial \Psi_{\alpha_\nu} \to \partial
\Psi_\alpha$ in the graph sense in $L^2(0,T;H)$, see \eqref{eq:attouch3}, the extension \eqref{eq:attouch} of~\cite[Proposition~3.59, p.~361]{Attouch} guarantees that $v\in \partial \Psi_\alpha(u)$, which
is nothing but \eqref{eq:prob2}.

A second consequence of \eqref{eq:tolin}, together with the bound \eqref{eq:boundsn} and convergence \eqref{eq:uweakee},
is that
\begin{align*}
  &\Psi_{\alpha}(u) \leq
    \liminf_{\nu \to 0} \Psi_{\alpha}(u_{\nu}) \leq  \limsup_{\nu \to
    0}  \Psi_{ \alpha}(u_{\nu}) \leq \limsup_{\nu \to
    0} \Psi_{\alpha_\nu}(u_{\nu}) \\
  &\quad\leq  \limsup_{\nu \to 0}\left( \Psi_{\alpha_\nu}(u) + \int_0^T (  v_{\nu },  u_{\nu }-u )\, \de t
    \right) \\
  &\quad =\limsup_{\nu \to 0}\left( \Psi_{\alpha}(u) +  \frac \nu 2  \|u\|^2_{L^2(0,T;H)}+ \int_0^T \left(  v_{\nu },  u_{\nu }-u \right) \, \de t
    \right)  =
    \Psi_{\alpha}(u).
\end{align*}
This proves that $\Psi_{\alpha}(u_{\nu })\to \Psi_{\alpha}(u)$.
As $\haz \alpha$ is strictly convex by \eqref{ass2},  \cite[Theorem~3 (ii)]{Visintin} and convergence \eqref{eq:uweakee}  ensure that 
\begin{equation*}
  u_{\nu} \to u \quad \text{in} \ L^{ \sigma}(0,T;H)
\end{equation*}
for any $ \sigma \in [1,2)$.  In particular,  putting $\sigma=q' \in (1,2)$ (indeed, $q > 2$),  we obtain 
$$\int_0^T (w_\nu,u_\nu)\, \de t \to \int_0^T (w,u)\,
\de t.$$
 As convergences \eqref{eq:uweakee} and \eqref{eq:gammaweakee}
 also  imply weak convergence in $L^2(0,T;H)$ the latter  
allows us to apply~\cite[Proposition~2.5, p.~27]{Brezis73} and obtain the inclusion
$w \in \partial \Psi_\beta(u)$,
that is \eqref{eq:prob4}.
Moreover,  we deduce that  $G(u_\nu) \to G(u)$ in $L^{  \sigma  }(0,T;H)$ for every $  \sigma  \in [1,2)$  due to \eqref{ass_lip0}.  Using the linear maximal monotonicity and $m$-accretivity of $A$ and $\mathcal B$ with $X = V^*$ and $p = 2$ (see \S \ref{sec:main14}), respectively, 
one can pass to the limit in \eqref{eq:prob5n} as $\epsi \to 0$  and obtain
\eqref{eq:prob5}. This concludes the proof of Theorem \ref{thm:main}.

\medskip
\section*{Acknowledgments} 

This research was funded in whole or in part by the Austrian Science Fund (FWF) projects 10.55776/F65, 10.55776/I5149, 10.55776/P32788, as well as by the OeAD-WTZ project CZ 09/2023. For open-access purposes, the authors have applied a CC BY public copyright license to any author-accepted manuscript version arising from this submission.  G.A.~is supported by JSPS KAKENHI Grant Numbers JP24H00184, JP21KK0044, JP21K18581, JP20H01812 and JP20H00117.  Part of this research was conducted during a visit to the Mathematical Institute at Tohoku University, whose warm hospitality is gratefully acknowledged.

\bibliographystyle{plain}

\begin{thebibliography}{1}

\bibitem{Affili}
E.~Affili, E.~Valdinoci.
\newblock Decay estimates for evolution equations with classical and fractional time-derivatives. 
\newblock {\it J.~Differential Equations} 266 (2019), 4027--4060.
  
\bibitem{Aizicovici-Hokkanen04b}
S.~Aizicovici, V.-M.~Hokkanen.
\newblock Doubly nonlinear equations with unbounded operators.
\newblock {\em Nonlinear Anal.} 58 (2004), no.~5-6:591--607.

\bibitem{Aizicovici-Hokkanen04}
S.~Aizicovici, V.-M.~Hokkanen.
\newblock Doubly nonlinear periodic problems with unbounded operators.
\newblock {\em J.~Math.~Anal.~Appl.} 292 (2004), no.~2:540--557.

\bibitem{Akagi06} 
G.~Akagi.
\newblock Doubly nonlinear evolution equations governed by time-dependent subdifferentials in reflexive Banach spaces.
\newblock {\em J.~Differential Equations} 231 (2006), no.~1, 32--56.

\bibitem{Akagi11}
G.~Akagi. 
\newblock Doubly nonlinear evolution equations with non-monotone perturbations in reflexive Banach spaces. 
\newblock {\it J.~Evol.~Equ.} 11 (2011), no.~1, 1--41. 

\bibitem{Akagi19}
G.~Akagi.
\newblock Fractional flows driven by subdifferentials in {H}ilbert spaces.
\newblock {\em Israel J.~Math.} 234 (2019), no.~2, 809--862.

\bibitem{Akagi24}
G.~Akagi, Y.~Nakajima.
\newblock Time-fractional gradient flows for nonconvex energies in {H}ilbert spaces.
\newblock {\tt arXiv:2501.08059 [math.AP]}.

\bibitem{Akagi-Schimperna}
G.~Akagi, G.~Schimperna. 
\newblock On a class of doubly nonlinear evolution equations in Musielak-Orlicz spaces. 
\newblock {\it Math.~Nachr.} 297 (2024), no.~7, 2686--2729.
  
\bibitem{Akagi14} G.~Akagi, U.~Stefanelli.
\newblock Doubly nonlinear equations as convex minimization. 
\newblock {\it SIAM J.~Math.~Anal.} 46 (2014), no.~3, 1922--1945.
 
\bibitem{Alt-Luckhaus}
H.W.~Alt, S.~Luckhaus.
\newblock Quasilinear elliptic-parabolic differential equations.
\newblock {\em Math.~Z.} 183 (1983), no.~3, 311--341.

\bibitem{Attouch}
H.~Attouch. 
\newblock {\it Variational convergence for functions and operators}. 
Applicable Mathematics Series. Pitman (Advanced Publishing Program), Boston, MA, 1984.

 \bibitem{B-dn} V.~Barbu.
\newblock{Existence for nonlinear Volterra equations in Hilbert spaces}. 
\newblock{\em SIAM J.~Math.~Anal.} 10 (1979), 552--569.
         
\bibitem{Bernis88}
F.~Bernis.
\newblock Existence results for doubly nonlinear higher order parabolic equations on unbounded domains.
\newblock {\em Math.~Ann.} 279 (1988), no.~3, 373--394.

\bibitem{Boegelein}
V.~B\"ogelein, M.~Strunk. 
\newblock A comparison principle for doubly nonlinear parabolic partial differential equations. 
\newblock {\it Ann.~Mat.~Pura Appl.~(4)} 203 (2024), no.~2, 779--804. 

\bibitem{BII}
M.~Bonforte, M.~Gualdani, P.~Ibarrondo.
\newblock Time-fractional porous medium type equations: sharp time decay and regularization.
{\tt arXiv:2401.03234 [math.AP]}

\bibitem{Brezis73}
H.~Br{\'e}zis.
\newblock {\em Op\'erateurs maximaux monotones et semi-groupes de contractions
  dans les espaces de {H}ilbert}.
\newblock North-Holland Publishing Co., Amsterdam, 1973.
\newblock North-Holland Mathematics Studies, no.~5. Notas de
Matem\'atica (50).

\bibitem{DiBenedetto-Show81}
E.~DiBenedetto, R.E.~Showalter.
\newblock Implicit degenerate evolution equations and applications.
\newblock {\em SIAM J.~Math.~Anal.} 12 (1981), no.~5, 731--751.

\bibitem{Dipierro}
S.~Dipierro, E.~Valdinoci, V.~Vespri.
\newblock Decay estimates for evolutionary equations with fractional time-diffusion. 
\newblock {\it J.~Evol.~Equ.} 19 (2019), 435--462.

\bibitem{Gajewsky-Skrypnik04}
H.~Gajewski, I.V.~Skrypnik.
\newblock To the uniqueness problem for nonlinear parabolic equations.
\newblock {\em Discrete Contin.~Dyn.~Syst.} 10 (2004), no.~1-2, 315--336.

\bibitem{Giga}
Y.~Giga, T.~Namba. 
\newblock Well-posedness of Hamilton--Jacobi equations with Caputo’s time-fractional derivative.
\newblock {\it Comm.~Partial Differential Equations}, 42 (2017), 1088--1120.

\bibitem{Grange-Mignot72}
O.~Grange, F.~Mignot.
\newblock Sur la r\'esolution d'une \'equation et d'une in\'equation paraboliques non lin\'eaires.
\newblock {\em J.~Funct.~Anal.} 11 (1972), 77--92.

\bibitem{GLS} 
G.~Gripenberg, S.-O.~Londen, O.~Staffans.
{\it Volterra integral and functional equations}.
Encyclopedia Math.~Appl., vol.~34, 
Cambridge University Press, Cambridge, 1990. %, xxii+701 pp.


\bibitem{Hokkanen91}
V.-M.~Hokkanen.
\newblock An implicit nonlinear time dependent equation has a solution.
\newblock {\em J.~Math.~Anal.~Appl.} 161 (1991), no.~1, 117--141.

\bibitem{Hokkanen92b}
V.-M.~Hokkanen.
\newblock Existence for nonlinear time dependent {V}olterra equations in {H}ilbert spaces.
\newblock {\em An.~\c Stiin\c t.~Univ.~Al.~I.~Cuza Ia\c si Sec\c t.~I a Mat.} 38 (1992), no.~1, 29--49.

\bibitem{Hokkanen92}
V.-M.~Hokkanen.
\newblock On nonlinear {V}olterra equations in {H}ilbert spaces.
\newblock {\em Differential Integral Equations} 5 (1992), no.~3, 647--669.

\bibitem{Huaroto}
G.~Huaroto. 
\newblock Doubly nonlinear isotropic fractional parabolic equation. 
\newblock {\it J.~Elliptic Parabol.~Equ.} 10 (2024), no.~2, 915--941.

\bibitem{Kato}
N.~Kato, M.~Misawa, K.~Nakamura, Y.~Yamaura.
\newblock Existence for doubly nonlinear fractional $p$-Laplacian equations. 
\newblock {\it Ann.~Mat.~Pura Appl.~(4)} 203 (2024), no.~6, 2481--2527. 

\bibitem{LiS23} 
W.~Li, A.J.~Salgado.
\newblock Time fractional gradient flows: theory and numerics.
\newblock {\it Math.~Models Methods Appl.~Sci.} {\bf 33} (2023), no.~2, 377--453.

 \bibitem{MW-dn} 
E.~Maitre, P.~Witomski.
\newblock{A pseudo-monotonicity adapted to doubly nonlinear elliptic-parabolic equations}.
 \newblock {\em  Nonlinear Anal.} 50 (2002), 223--250.

\bibitem{Moring}
 K.~Moring, L.~Sch\"atzler, C.~Scheven. 
\newblock Higher integrability for singular doubly nonlinear systems. 
\newblock {\it Ann.~Mat.~Pura Appl.~(4)} 203 (2024), no.~5, 2235--2274. 

\bibitem{Ploci1}
\L.~P\l ociniczak. 
\newblock Approximation of the Erd\'elyi–Kober operator with application to the time-fractional porous medium equation. 
\newblock {\it SIAM J.~Appl.~Math.} 74 (2014), 1219--1237.

\bibitem{Ploci2}
\L.~P\l ociniczak. 
\newblock Analytical studies of a time-fractional porous medium equation. Derivation, approximation and applications. 
\newblock {\it Commun.~Nonlinear Sci.~Numer.~Simul.} 24 (2015), 169--183.

\bibitem{Ploci3}
\L.~P\l ociniczak, H.~Okrasi\'nska. 
\newblock Approximate self-similar solutions to a nonlinear diffusion equation with time-fractional derivative. 
\newblock {\it Phys.~D} 261 (2013), 85--91.

\bibitem{vol}
U.~Stefanelli.
\newblock Well-posedness and time discretization of a nonlinear {V}olterra integrodifferential equation.
\newblock {\em J.~Integral Equations Appl.} 13 (2001), no.~3, 273--304.

\bibitem{fico}
U.~Stefanelli.
\newblock On a class of doubly nonlinear nonlocal evolution equations.
\newblock {\em Differential Integral Equations} 15 (2002), no.~8, 897--922.

\bibitem{ganzo}
U.~Stefanelli.
\newblock On some nonlocal evolution equations in {B}anach spaces.
\newblock {\em J.~Evol.~Equ.} 4 (2004), no.~1, 1--26.


\bibitem{be}
U.~Stefanelli. 
\newblock The Brezis-Ekeland principle for doubly nonlinear equations.
\newblock {\it SIAM J.~Control Optim.} 47 (2008), 1615--1642. 

\bibitem{Topp}
E.~Topp, M.~Yangari.
\newblock Existence and uniqueness for parabolic problems with Caputo time derivative.
\newblock {\it J.~Differential Equations} 262 (2017), 6018--6046.

\bibitem{Vergara1}
V.~Vergara, R.~Zacher.
\newblock A priori bounds for degenerate and singular evolutionary partial integro-differential equations.
\newblock {\it Nonlinear Anal.} 73 (2010), 3572--3585.

\bibitem{Vergara2} 
V.~Vergara, R.~Zacher. 
\newblock Optimal decay estimates for time-fractional and other nonlocal subdiffusion equations via energy methods. 
\newblock {\it SIAM J.~Math.~Anal.} 47 (2015), 210--239.
   
\bibitem{Visintin}
A.~Visintin.
\newblock Strong convergence results related to strict convexity.
\newblock {\em Comm.~Partial Differential Equations} 9 (1984), no.~5, 439--466.

\bibitem{Visintin96}
A.~Visintin. 
\newblock \emph{Models of phase transitions}, Progress in Nonlinear
  Differential Equations and Their Applications, vol.~28,
  Birkh\"auser, 1996.

\bibitem{Za09} 
R.~Zacher,
\newblock{Weak solutions of abstract evolutionary integro-differential equations in Hilbert spaces}.
\newblock{\em Funkcial.~ Ekvac.} 52 (2009), 1--18.
 %%

\end{thebibliography}

\def\cprime{$'$}

\end{document}